\documentclass[12pt]{amsart}
\usepackage[left=1in,top=1in,right=1in,bottom=1in,letterpaper]{geometry}
\usepackage{epsfig}
\usepackage{amsmath}
\usepackage{amssymb}
\usepackage{amscd}
\usepackage{url}
\usepackage{verbatim} 
\usepackage{color}
\usepackage{epsfig}
\usepackage{amsthm}
\usepackage{pifont}

\usepackage{hyperref}

\title{Grid graphs and lattice surfaces}
\author{W. Patrick Hooper}
\thanks{Support was provided by N.S.F. Postdoctoral Fellowship DMS-0803013, N.S.F. Grant DMS-1101233 and a PSC-CUNY Award (funded by The Professional Staff Congress and The City University of New York).}
\address{
The City College of New York\\
New York, NY, USA 10031}
\email{whooper@ccny.cuny.edu}
\date{\today}

\newtheorem{theorem}{Theorem}
\newtheorem{proposition}[theorem]{Proposition}
\newtheorem{lemma}[theorem]{Lemma}
\newtheorem{remark}[theorem]{Remark}
\newtheorem{corollary}[theorem]{Corollary}

\newtheorem{question}[theorem]{Question}

\theoremstyle{definition}

\def\C{\mathbb{C}}%
\def\N{\mathbb{N}}%
\def\P{\mathbb{P}}%
\def\Q{\mathbb{Q}}%
\def\R{\mathbb{R}}%
\def\Z{\mathbb{Z}}%
\def\GL{\textit{GL}}
\def\SL{\textit{SL}}
\def\SO{\textit{SO}}
\def\PGL{\textit{PGL}}
\def\PSL{\textit{PSL}}

\renewcommand{\d}[1]{\ensuremath{\textit{d$#1$}}}%
\def\Moduli{{\mathcal M}} 

\def\dev{\textit{dev}}%
\def\H{\mathbb H}%

\def\isomS1{\textit{Isom}_+(S^1)}
\def\rt3{\sqrt{3}}

\def\tr{\mathrm{tr}} 

\def\T{\mathcal{T}}

\def\Isom{\mathrm{Isom}}

\newtheorem{convention}{Convention}

\makeatletter
\def\@strippedMR{}
\def\@scanforMR#1#2#3\endscan{%
  \ifx#1M\ifx#2R\def\@strippedMR{#3}%
  \else\def\@strippedMR{#1#2#3}%
  \fi\fi}

\def\Alpha{{\mathcal A}}
\def\Beta{{\mathcal B}} 
\def\n{{\mathfrak n}} 
\def\e{{\mathfrak e}} 
\def\E{{\mathcal E}} 
\def\G{{\mathcal G}} 
\def\H{{\mathbb H}} 
\def\V{{\mathcal V}} 

\def\Aff{\textit{Aff}}
\def\til{\widetilde}
\def\sC{{\mathcal C}}

\newcommand{\compat}[1]{{\textcolor{red}{\bf #1}}}

\begin{document}
\begin{abstract}
First, we apply Thurston's construction of pseudo-Anosov homeomorphisms to grid graphs and obtain translation surfaces whose Veech groups are commensurable to $(m,n,\infty)$ triangle groups. These surfaces were first discovered by Bouw and M\"oller, however our treatment of the surfaces differs. We construct these surfaces by gluing together polygons in two ways. We use these elementary descriptions to compute the Veech groups, resolve primitivity questions, and describe the surfaces
algebraically. Second, we show that some $(m,n, \infty)$ triangle groups can not arise as Veech groups. This generalizes work of Hubert and Schmidt.
\end{abstract}
\maketitle

\section{Introduction}

A {\em translation surface} $(X,\omega)$ is a closed Riemann surface $X$ equipped with a non-zero holomorphic $1$-form $\omega$. 
There is a well known action of $\GL(2,\R)$ on the moduli space of all translation surfaces. 
We use $\GL(X,\omega) \subset \GL(2,\R)$ to denote the subgroup of all $A \in \GL(2, \R)$ for which $A(X, \omega)=(X,\omega)$.
By area considerations, the determinant of an element $A \in \GL(X, \omega)$ must be $\pm 1$. 
The {\em Veech group} of $(X, \omega)$ is the group $\SL(X, \omega)=\GL(X, \omega) \cap \SL(2, \R)$. We define 
$\PGL(X, \omega)$ and $\PSL(X, \omega)$ to be the projections of these groups to $\PGL(2, \R)$ and $\PSL(2, \R)$, respectively.
We say $(X,\omega)$ has the {\em lattice property} if the Veech group has finite co-volume in $\SL(2,\R)$. 

Interest in these objects is sparked by connections with Teichm\"uller theory. See \cite[\S 2.3]{MT}, for instance.
If $(X, \omega)$ is a translation surface, there is a totally geodesic isometric immersion of $\H^2/\PSL(X, \omega)$ into ${\mathcal M}_g$, the
moduli space of surfaces of genus $g=\textit{genus}(X)$ equipped with the Teichm\"uller metric. 

The {\em $(m,n, \infty)$ triangle group} is the group 
$$\langle a, b, c~:~a^2=b^2=c^2=(ac)^m=(bc)^n=e \rangle.$$
This group can be realized as a subgroup $\Delta(m,n, \infty) \subset \PGL(2, \R)=\Isom(\H^2)$
generated by reflections in the sides of a hyperbolic triangle with one ideal vertex and two angles of $\pi/m$ and $\pi/n$.
We use $\Delta^+(m,n,\infty)$ to refer to the orientation preserving part,
$\Delta(m,n,\infty) \cap \PSL(2, \R)$. 

We will describe translation surfaces $(X, \omega)$ for which the group $\PGL(X, \omega)$ is conjugate to $\Delta(m,n, \infty)$ or an index two subgroup of $\Delta(m,n, \infty)$. 
In one sentence, these surfaces are constructed by applying Thurston's construction of pseudo-Anosov homeomorphisms to grid graphs. 
Sections \ref{sect:thurston_veech} and \ref{sect:results} explain.
We reprove the following theorem of Bouw and M\"oller \cite{BM}.

\begin{theorem}[Veech triangle groups]
\label{thm:veech_triangle_groups}
Let $m$ and $n$ be integers satisfying $2 \leq m < n < \infty$. 
\begin{itemize}
\item
If $m$ and $n$ are not both even, then there is a translation surface for which
$\PGL(X, \omega)$ is conjugate to $\Delta(m,n, \infty)$. 
\item When $m$ and $n$ are even, there is a translation surface for which
$\PGL(X, \omega)$ is conjugate to an index two subgroup of $\Delta(m,n, \infty)$. 
In this case, $\H^2/\PSL(X, \omega)$ is a hyperbolic $(m/2,n/2,\infty, \infty)$-orbifold. 
\end{itemize}
\end{theorem}
A more detailed restatement of this theorem which also covers the case $m=n$ is provided by theorem \ref{thm:veech_groups}.
The surfaces we describe are the same as the surfaces constructed in \cite{BM}. (See Corollary \ref{cor:same} and the discussion below this corollary.)
What sets this work apart is a concrete treatment of surfaces with these Veech groups.
By applying Thurston's construction to grid graphs, we obtain a description of these surfaces by gluing together rectangles. 
We also provide a description of the surfaces in terms of a Riemann surface and a holomorphic $1$-form in all cases.
This was provided in \cite{BM} in the cases when $m$ and $n$ are relatively prime.

Our concrete treatment of these surfaces enabled us to discover finer information about these surfaces. See further below. Perhaps the following is more surprising.

\begin{theorem}[Non-Veech triangle groups]
\label{thm:non_veech_triangle_groups}
Let $m$ and $n$ be even integers, and let $\gamma=\gcd(m,n)$.
Under either the assumption that $m/\gamma$ and $n/\gamma$ are both odd or that $\gamma=2$, 
there is no translation surface for which $\Delta^+(m,n,\infty) \subset \PSL(X, \omega)$.
\end{theorem}
This theorem contradicts a claim that appeared in a version of \cite{BM}. 
Hubert and Schmidt remarked a proof of this theorem in the special case when $m=2$ \cite[remark 7]{HS01}.

The following seems to be a very interesting open question.

\begin{question}
For $m$ and $n$ even and not satisfying the conditions of theorem \ref{thm:non_veech_triangle_groups}, is there a translation surface 
for which $\Delta^+(m,n,\infty) \subset \PSL(X, \omega)$? Can the $(m,n,\infty)$ orbifold be isometrically immersed in ${\mathcal M}_g$
for some $g$?
\end{question}

We now return to a discussion of the translation surfaces constructed to prove theorem \ref{thm:veech_triangle_groups}. 
New examples of surfaces with the lattice property can be constructed by carefully taking branched covers of
existing examples. 
A translation surface is {\em primitive} if it does not arise as a branched covering of a translation surface of smaller genus. 
A lattice $\Gamma \subset \SL(2, \R)$  is {\em arithmetic} if it contains a finite index subgroup which can be
conjugated into $\SL(2, \Z)$. If $\SL(X, \omega)$ is arithmetic, then $(X, \omega)$ is a cover of a torus, branched at one point. See theorem \ref{thm:gutkin-judge} and \cite{GJ00}.

As mentioned above, we will explicitly construct translation surfaces for which $\PGL(X, \omega)$ is commensurable to $\Delta(m,n, \infty)$. 
We will explicitly compute the Veech group of these surfaces, and show that these surfaces are primitive whenever their Veech groups are non-arithmetic. 
There are only three arithmetic $(m,n,\infty)$ triangle groups with $2 \leq m < n < \infty$.
In \cite{BM}, it was shown that these surfaces are primitive when $m$ and $n$ are relatively prime.

\subsection{Structure of paper}

We give some background on the problems described above in the next section. The major tool we use is the Thurston's construction of pseudo-Anosov homeomorphisms,
which we describe in \S \ref{sect:thurston_veech}. 

We give concrete descriptions of translation surfaces which meet the criteria of theorem \ref{thm:veech_triangle_groups} in \S \ref{sect:results}.
In \S \ref{sect:bm}, we apply the Thurston's construction to grid graphs to build these translation surfaces. 
Questions about topology, the Veech groups, arithmeticity, and primitivity are also answered in this section.
In \S \ref{sect:semiregular_background}, we describe affinely equivalent surfaces built from ``semiregular polygons.'' 
This description of these surfaces was independently discovered by Ronen Mukamel.
This point of view gives us algebraic formulas for the Riemann surface and holomorphic $1$-form in each case.
Sections \ref{sect:semiregular} through \ref{sect:formulas} are devoted to proving these facts. We outline the 
structure of these sections in \S \ref{sect:outline}.

We prove theorem \ref{thm:non_veech_triangle_groups} in section \ref{sect:not_veech}. 

\section{Background}
\label{sect:1}
\label{sect:background}

\subsection{Translation surfaces and the lattice property}
A {\em translation surface} $(X,\omega)$ is a Riemann surface $X$ equipped with a non-zero holomorphic $1$-form, $\omega$. The $1$-form $\omega$ 
provides local charts from $X$ to $\C$ defined up to translation, away from the zeros of $\omega$. 
At a zero, we have a chart to the Riemann surface $w=z^{k+1}$, where $k$ is the order of the zero.
A translation surface inherits a singular Euclidean metric by pulling back the Euclidean metric via the charts. 
From this point of view, a zero of order $k$ of $\omega$ yields
a cone singularity with cone angle $2 \pi (k+1)$.
In particular, we may equivalently think of
a translation surface as a finite union of polygonal subsets of $\C$ with edges glued in pairs by translations. 
Cone singularities may appear at the equivalence class of a vertex of a polygon.
Also, a translation surface inherits a notion of {\em direction} from $\C$. This is just the map which sends a tangent vector on $X$ to its image vector in $\C$ under a chart.
Note that the direction of a tangent vector is invariant under the geodesic flow in this metric.

Let $(X,\omega)$ and $(X_0,\omega_0)$ be translation surfaces. A homeomorphism $f:X \to X_0$ is called {\em affine} if it preserves the underlying affine structures. That is, there are real numbers $a,b,c,d$
such that on local charts
$f(x+iy)=(ax+by)+i(cx+dy)$.
We use $f:(X, \omega) \to (X_0, \omega_0)$ to denote such a map.
An {\em affine automorphism} of $(X,\omega)$ is an affine homeomorphism $f:(X, \omega) \to (X, \omega)$. 
Since $f$ must preserve area we have $ad-bc=\pm 1$. The collection of all affine automorphisms forms the {\em affine automorphism group}, $\Aff(X,\omega)$. 
The {\em derivative map} $D:\Aff(X, \omega) \to \GL (2,\R)$ recovers the action of an affine automorphism on local charts. That is,
$$D:f \mapsto \left[\begin{array}{rr} a & b \\ c & d \end{array}\right].$$
We use $\GL(X, \omega)$ to denote $D\big(\Aff(X,\omega)\big) \subset \GL(2, \R)$.
The orientation preserving part of $\GL(X, \omega)$
is the {\em Veech group}, $\SL(X, \omega)=D\big(\Aff(X,\omega)\big)=\SL(2,\R) \cap \GL(X,\omega)$. 
As in the introduction, we use $\PGL(X, \omega)$ and $\PSL(X, \omega)$ to denote the projectivizations
of these groups.

It is a theorem of Veech that $\SL(X,\omega)$ is discrete and is never co-compact \cite{V}. We say that $(X,\omega)$ has the {\em lattice property}
if $\SL(X,\omega)$ has finite co-volume in $\SL(2,\R)$. 

Interest in the lattice property arises from Teichm\"uller theory. Consider the hyperbolic plane, 
${\mathbb H}^2=\SO(2,\R) \setminus \SL(2,\R)$. 
The quotient ${\mathbb H}^2/\PSL(X,\omega)$ naturally immerses into the moduli space of complex structures on a surface with the genus of $X$. 
In fact, if we consider the related topic of {\em half-translation surfaces} (Riemann surfaces paired with a holomorphic quadratic differential),
every complete, finite area, totally geodesic subsurface of $\Moduli_g$ is isometric to ${\mathbb H}^2/\PSL(X,\omega)$,
for some naturally related half-translation surface $(X,\omega)$ with the lattice property \cite{MT}.
Veech found the first examples of translation surfaces with the lattice property \cite{V}.
Since then, there has been interest in finding more examples and classifying these objects. 
See \cite{KS}, \cite{McM03}, \cite{Calta04}, \cite{BM}, and \cite{McM06} for instance.

Let $(X,\omega)$ and $(X_0, \omega_0)$ be translation surfaces. 
A {\em covering} is a surjective (possibly branched) holomorphic map $f:X \to X_0$ for which the pullback form $f^\ast(\omega_0)$ equals $\omega$. 
$(X,\omega)$ is called {\em primitive} if it does not cover a surface of smaller genus. 
The following theorem of M\"oller answered a question of Hubert and Schmidt in \cite{HS01}. See
\cite[theorem 2.6]{Mo06} or \cite[\S 2]{McM06}. 

\begin{theorem}[M\"oller]
\label{thm:primitive}
Every translation surface $(X,\omega)$ covers a primitive translation surface $(X_0, \omega_0)$. 
If the genus of $X_0$ is greater than $1$, then this 
covering is unique and $\SL(X,\omega)$ is a subgroup of $\SL(X_0, \omega_0)$. 
\end{theorem}

We now concentrate on the special case when $(X, \omega)$ covers a torus $(X_0, \omega_0)$. Two subgroups $\Gamma_1, \Gamma_2 \subset \SL(2,\R)$ 
are {\em commensurable} if there are finite index subgroups $G_1 \subset \Gamma_1$ and $G_2 \subset \Gamma_2$ which are conjugate in $\SL(2, \R)$. 
A subgroup of $\SL(2, \R)$ is {\em arithmetic} if it is commensurable to $\SL(2, \Z)$. We have the following theorems about surfaces with arithmetic
Veech groups.

\begin{theorem}[Gutkin-Judge {\cite{GJ00}}]
\label{thm:gutkin-judge}
Let $(X,\omega)$ be a translation surface with the lattice property. The following are equivalent.
\begin{enumerate}
\item $(X,\omega)$ is a branched cover of a torus. 
\item $\SL(X, \omega)$ is arithmetic.
\end{enumerate}
\end{theorem}

\begin{theorem}[Schmith\"usen {\cite{Schm04}}]
\label{thm:schmithusen}
If $\SL(X,\omega)$ is arithmetic, then $\SL(X, \omega)$ is conjugate to a subgroup of $\SL(2, \Z)$.
\end{theorem}

Theorem \ref{thm:non_veech_triangle_groups} can be viewed a generalization of the following observation.

\begin{corollary}[Some non-Veech triangle groups]
The orientation preserving parts of the triangle groups $\Delta(2,4, \infty)$, $\Delta(2, 6, \infty)$, $\Delta(4,4, \infty)$ and $\Delta(6,6, \infty)$ can
not be subgroups of $\PSL(X, \omega)$ for any translation surface $(X, \omega)$.
\end{corollary}
\begin{proof}
The arithmetic triangle groups are classified in \cite{Tak77}, and include the listed triangle groups. 
These groups are not conjugate into $\SL(2, \Z)$ (or even $\SL(2, \Q)$), because they contain elements with irrational trace.
\end{proof}
\section{Veech groups with non-commuting parabolics}
\label{sect:thurston_veech}
We will introduce {\em Thurston's construction}, a combinatorial construction which produces a translation surface $(X_0, \omega_0)$ 
with hyperbolic elements in $\SL(X,\omega)$. Thurston used this construction to generate pseudo-Anosov automorphisms of surfaces \cite[theorem 7]{T88}.
It follows from work of Veech that all translation surfaces with the lattice property arise from this construction \cite[\S 9]{V}.
McMullen realized that a concise way to describe the combinatorics of this construction is via a bipartite ribbon graph \cite[\S 4]{McM06}.

A {\em bipartite ribbon graph} is a finite connected graph $\G$ with vertex set $\V$ and edge set $\E$, equipped with two permutations $\n, \e:\E \to \E$, that satisfy the following conditions.
\begin{itemize}
\item The vertex set $\V$ is a disjoint union of two sets $\Alpha$ and $\Beta$.
\item There are functions $\alpha:\E \to \Alpha$ and $\beta:\E \to \Beta$ such that every edge $e \in \E$ joins vertex $\alpha(e) \in \Alpha$ to the vertex $\beta(e)\in \Beta$. 
\item For all $e \in \E$, the orbit ${\mathcal O}_\e(e)=\{\e^k(e)~:~k \in \N\}$ satisfies $\alpha({\mathcal O})=\alpha(e)$. Similarly, 
the orbit ${\mathcal O}_\n(e)=\{\n^k(e)~:~k \in \N\}$ satisfies $\beta({\mathcal O})=\beta(e)$.
\end{itemize}
The first two statements make $\G$ a bipartite graph. The third statement says that the cycles of the permutations $\e$ and $\n$ are the 
edges with a common vertex in $\Alpha$ or a common vertex in $\Beta$, respectively. These cycles make $\G$ an oriented ribbon graph. 

Given the data above plus a function $w:\V \to \R_{>0}$, we can construct a translation surface $(X_{\G,w}, \omega_{\G,w})$. This surface is a union of the rectangles
$R_e$ for $e \in \E$ with 
$$R_e=[0, w \circ \beta(e)] \times [0, w \circ \alpha(e)].$$
To build $(X_{\G,w}, \omega_{\G,w})$ isometrically identify the right side of each rectangle $R_e$ to the left side of $R_{\e(e)}$ and the top side of $R_e$ to the bottom side of $R_{\n(e)}$.
(This justifies the choice of symbols for the permutations; $\e$ is for east and $\n$ is for north.)

We now review some standard definitions.
A {\em saddle connection} in a translation surface $(X, \omega)$ is a geodesic segment $\sigma$ 
which intersects the zeros of $\omega$ precisely at its endpoints. A {\em cylinder} in $(X, \omega)$ is
a closed subset isometric to a Euclidean cylinder of the form $[0,a] \times \R/k\Z$.
The positive real constants $a$ and $k$ are the {\em width} and {\em circumference} of the cylinder, respectively.
The interior of a cylinder in $(X, \omega)$ is foliated by periodic trajectories of the geodesic flow. 
The direction of each trajectory viewed as an element of $\R\P^1$ is the same, and we call this the {\em direction} of the cylinder.
Each boundary component of a cylinder is either a finite union of saddle connections or a periodic trajectory.
A {\em cylinder decomposition} is finite collection of cylinders with disjoint interiors that cover the translation surface.
Each cylinder in a decomposition has the same direction, so we call this the {\em direction of the cylinder decomposition}.

The surface $(X_{\G,w}, \omega_{\G,w})$ comes equipped with both a horizontal and a vertical cylinder decomposition. These horizontal and vertical cylinders are in bijective correspondence
with the sets $\Alpha$ and $\Beta$ respectively. Given $a \in \Alpha$ and $b \in \Beta$, the respective cylinders are
$$\bigcup_{e \in \alpha^{-1}(a)} R_e 
\quad \textrm{and} \quad
\bigcup_{e \in \beta^{-1}(b)} R_e.$$

We say that $w$ is an {\em eigenfunction} of $\G$ corresponding to the {\em eigenvalue} $\lambda \in \R$ if for all $x \in \V$
$$\sum_{\overline{xy} \in \E} w(y)=\lambda w(x).$$
If $w$ is a positive eigenfunction with eigenvalue $\lambda$, then the Veech group of $(X_{\G,w}, \omega_{\G,w})$ contains the elements
\begin{equation}
\label{eq:standard_matrices}
P_0=\left[\begin{array}{rr} 1 & \lambda \\ 0 & 1 \end{array}\right] 
\quad
\textrm{and}
\quad
Q_0=\left[\begin{array}{rr} 1 & 0 \\ -\lambda & 1 \end{array}\right].
\end{equation}
(See \cite[\S 9]{V} and \cite[\S 4]{McM06}, for instance.)
Note, by the Perron-Frobeninus theorem, there is a unique positive eigenfunction up to scalar multiplication. 
For most applications, the choice of this eigenfunction is irrelevant. 
So, we will use $(X_\G,\omega_\G)$ to denote $(X_{\G,w}, \omega_{\G,w})$ where $w$ is a positive eigenfunction of the adjacency matrix of $\G$.

In \S \ref{sect:not_veech}, we will use the following consequence of comments in \cite[\S 9]{V}.

\begin{theorem}[Veech]
\label{thm:graphs}
Given any $(Y, \eta)$ such that $\SL(Y, \eta)$ contains two non-commuting parabolics $P$ and $Q$, there is
a bipartite ribbon graph $\G$ and an affine homeomorphism $\phi:(Y, \eta) \to (X_\G, \omega_\G)$. Moreover, we can assume that
$$
D(\phi) \circ P\circ D(\phi)^{-1} =\left[\begin{array}{rr} \pm 1 & r \lambda \\ 0 & \pm 1 \end{array}\right] 
\quad
\textrm{and}
\quad
D(\phi) \circ Q\circ D(\phi)^{-1}=\left[\begin{array}{rr} \pm 1 & 0 \\ s \lambda & \pm 1 \end{array}\right],
$$
where $\lambda$ is the Perron-Frobenius eigenvector of $\G$, $r,s \in \Q$ are non-zero, and the choice of signs depends
on the sign of the eigenvalues of $P$ and $Q$.
\end{theorem} 
\section{Veech triangle groups}
\label{sect:results}

In this section, we describe our construction of translation surfaces $(X_{m,n}, \omega_{m,n})$ 
for which $\PGL(X_{m,n}, \omega_{m,n})$ is commensurable to $\Delta(m,n,\infty)$. 
Section \ref{sect:outline} reveals where we prove these results.

\subsection{Grid graphs}
\label{sect:bm}

For integers $m$ and $n$ with $m \geq 2$ and $n \geq 2$, we define the
{\em $(m,n)$ grid graph} to be the graph $\G_{m,n}$ whose vertices are $v_{i,j}$ for integers $i$ and $j$ satisfying
$1 \leq i < m$ and $1 \leq j <n$. We define
$$\E=\{\overline{v_{i,j} v_{k,l}} ~:~ (i-k)^2+(j-l)^2=1\}.$$
(Equivalently, embed the vertices in $\R^2$ in the natural way as in figure \ref{fig:graph_example2} and join edges between vertices of distance $1$.)
We make this graph bipartite by defining the disjoint subsets $\Alpha, \Beta$ by
\begin{equation}
\label{eq:node_labels}
\Alpha=\{v_{i,j} \in \V~:~\textrm{$i+j$ is even}\} 
\quad
\textrm{and}
\quad
\Beta=\{v_{i,j} \in \V~:~\textrm{$i+j$ is odd}\} 
\end{equation}
We will use the notation $\alpha_{i,j}=v_{i,j}$ provided $v_{i,j} \in \Alpha$ and 
$\beta_{i,j}=v_{i,j}$ provided $v_{i,j} \in \Beta$.

\begin{figure}
\begin{center}
\includegraphics[height=1.9in]{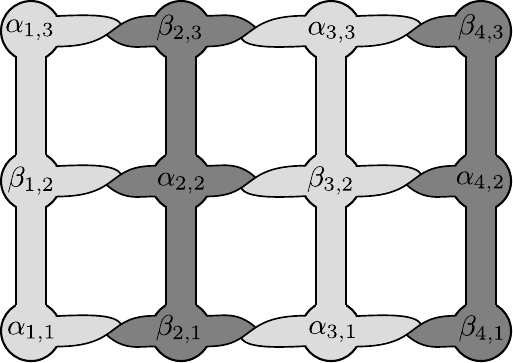}
\quad
\includegraphics[height=1.9in]{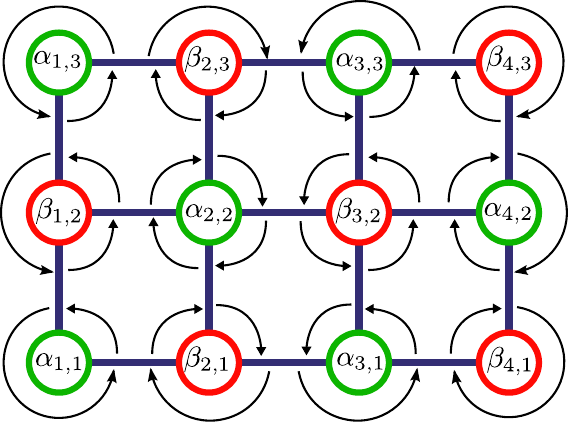}
\caption{The bipartite ribbon graph $\G_{5,4}$ is shown on the left. On the right, we view this as a graph
equipped with two edge permutations. The edge permutation $\e$ is indicated by the arrows surrounding the $\alpha$
vertices, and $\n$ is indicated by the arrows surrounding the $\beta$ vertices.}
\label{fig:graph_example2}
\end{center}
\end{figure}

So that the edge set $\E$ is non-empty, we make the assumption that $mn \geq 6$. 
We make $\G_{m,n}$ a ribbon graph by defining the permutations $\e,\n:\E \to \E$ according to the following convention.

\begin{convention}[Permutation convention]
\label{permutation_convention}
The permutations $\e,\n:\E \to \E$ are determined from cyclic orderings for the edges around each vertex $v_{i,j}$.
Consider the usual embedding of $\G_{m,n}$ into $\R^2$ as in figure \ref{fig:graph_example2}. We choose the clockwise ordering around $v_{i,j}$ 
when $i$ is even and the counter-clockwise ordering when $i$ is odd.
\end{convention}

The positive eigenfunction is given by the equation 
\begin{equation}
\label{eq:eigenfunction}
w(v_{i,j})=\sin (\frac{i \pi}{m}) \sin (\frac{j \pi}{n}).
\end{equation}
For $m$ and $n$ as above, define the translation surface $(X_{m,n}, \omega_{m,n})=(X_{\G_{m,n},w}, \omega_{\G_{m,n},w})$. 

The following matrices are all $2 \times 2$ matrices of determinant $-1$ with eigenvalues $\pm 1$. 
\begin{equation}
\label{eq:matrices}
A=\left[\begin{array}{rr}
-1 & -2 \cos \frac{\pi}{m} \\
0 & 1
\end{array}\right]
\quad 
B=\left[\begin{array}{rr}
-1 & 2 \cos \frac{\pi}{n} \\
0 & 1
\end{array}\right]
\quad
C=\left[\begin{array}{rr} 
0 & -1 \\ 
-1 & 0 
\end{array}\right]
\quad 
E=\left[\begin{array}{rr}
-1 & 0 \\
0 & 1
\end{array}\right]
\end{equation}
The projection of the subgroup $\langle A,B,C \rangle \subset \GL(2, \R)$ to $\PGL(2, \R)$ is conjugate to the triangle group $\Delta(m,n,\infty)$ described
in the introduction. The subgroup $\langle A,B,C \rangle \subset \GL(2, \R)$ is characterized by this image and the statement that $-I \in \langle A, B, C\rangle$. 
These matrices have the relations
$$A^2=B^2=C^2=I, \quad (AC)^m=-I, \quad \text{and} \quad (BC)^n=\begin{cases}
-I & \text{if $n$ is even}\\
I & \text{if $n$ is odd,}
\end{cases} $$
while $AB$ is parabolic. When $n=m$, $E$ appears as an additional symmetry, satisfying the relations
$E^2=I$, $EAE=B$ and $(EC)^2=-I$. The subgroup $\langle A,C,E \rangle \subset \GL(2, \R)$ projects to a group conjugate to $\Delta(2,m, \infty)$ in
$\PGL(2, \R)$. 

\begin{theorem}[The Veech groups]
\label{thm:veech_groups}
Let $m \geq 2$ and $n \geq 2$ be integers with $mn \geq 6$. 
If $mn<10$, then $(X_{m,n}, \omega_{m,n})$ is a torus 
and so $\GL(X_{m,n}, \omega_{m,n})$ is conjugate to $GL(2,\Z)$. 
The following statements determine the Veech group $\GL(X_{m,n}, \omega_{m,n})$
when $mn \geq 10$
\begin{itemize}
\item If $m \neq n$, then
\begin{itemize}
\item When $m$ and $n$ are not both even, $\GL(X_{m,n}, \omega_{m,n})=\langle A,B,C \rangle$.
\item When $m$ and $n$ are both even, $\GL(X_{m,n}, \omega_{m,n})=\langle A,B,CAC,CBC\rangle$. (This is an index two subgroup of
$\langle A,B,C\rangle$.) 
\end{itemize}
\item If $m=n$, then
\begin{itemize}
\item When $m$ is odd, $\GL(X_{m,m}, \omega_{m,m})=\langle A,C,E \rangle$.
\item When $m$ is even, $\GL(X_{m,m}, \omega_{m,m})=\langle A,E,CAC\rangle$. (This is a reflection group in an $(\frac{m}{2},\infty, \infty)$ triangle.)
\end{itemize}
\end{itemize}
\end{theorem}

\begin{corollary}[Arithmeticity]
The surface $(X_{m,n}, \omega_{m,n})$ has an arithmetic Veech group if and only if $(m,n)$ or $(n,m)$ is in the set
$\{(2,3), (2,4), (2, 6), (3,3), (4,4), (6,6)\}.$
\end{corollary}
\begin{proof}
It is straightforward to check that in all the listed cases the Veech group of $(X_{m,n}, \omega_{m,n})$ is conjugate to a subgroup of $\SL(2,\Z)$. 
We see that when $(m,n)$ or $(n,m)$ is $(4,6)$ then $CACB$ is in the Veech group, but its
trace is not an integer. In the remaining cases, the Veech group has elliptics whose trace is not rational.
\end{proof}

We will now consider primitivity. The graph $\G_{m,n}$ admits an automorphism $\iota$ defined by
\begin{equation}
\label{eq:iota}
\iota(v_{i,j})=v_{m-i,n-j}.
\end{equation}
In the special case that both $m$ and $n$ are even, this automorphism satisfies the following conditions.
\begin{itemize}
\item $\iota(\Alpha)=\Alpha$ and $\iota(\Beta)=\Beta$.
\item $\iota \circ \e=\e \circ \iota$ and $\iota \circ \n=\n \circ \iota$. 
\item $w \circ \iota=w$. 
\end{itemize}
These conditions imply that $\iota$ extends to an automorphism $\iota_\ast:(X_{m,n}, \omega_{m,n}) \to (X_{m,n}, \omega_{m,n})$, which simply permutes
the rectangles making up $(X_{m,n}, \omega_{m,n})$ according to $\iota$. In particular,
$(X_{m,n}^e, \omega_{m,n}^e)=(X_{m,n}, \omega_{m,n})/\iota_{\ast}$ is a translation surface covered by $(X_{m,n}, \omega_{m,n})$.

Note that $\SL(X^e_{m,n}, \omega^e_{m,n})$ is arithmetic if and only if $\SL(X_{m,n}, \omega_{m,n})$ is.

\begin{theorem}[Primitivity]
\label{thm:primitivity}
When $m$ and $n$ are not both even, $(X_{m,n}, \omega_{m,n})$ is primitive unless $\SL(X_{m,n}, \omega_{m,n})$
is arithmetic.
When both $m$ and $n$ are even, $(X_{m,n}^e, \omega_{m,n}^e)$ is primitive unless $\SL(X_{m,n}^e, \omega_{m,n}^e)$ is arithmetic.
\end{theorem}

It remains to describe the Veech groups of $(X_{m,n}^e, \omega_{m,n}^e)$. 

\begin{theorem}
\label{thm:even_veech_groups}
Suppose that $m$ and $n$ are even and that $(X_{m,n}^e, \omega_{m,n}^e)$ is not a torus. Then, 
$\GL(X_{m,n}^e, \omega_{m,n}^e)=\GL(X_{m,n}, \omega_{m,n})$.
\end{theorem}

\begin{theorem}[The stratum of $(X_{m,n}, \omega_{m,n})$]
\label{thm:topological_type}
Let $m \geq 2$ and $n \geq 2$ and assume $mn \geq 6$. Set $\gamma=\gcd(m,n)$. Then $X_{m,n}$ is a surface of genus
$\frac{mn-m-n-\gamma}{2}+1$.
Provided $m>3$ or $n>3$,
the $1$-form $\omega_{m,n}$ has $\gamma$ zeros, each of order $\frac{mn-m-n}{\gamma}-1$. 
\end{theorem}

\begin{theorem}[The stratum of $(X_{m,n}^e, \omega_{m,n}^e)$]
\label{thm:topological_type2}
If $m$ and $n$ are even integers satisfying $m \leq 4$ and $n \leq 4$, then $(X_{m,n}^e, \omega_{m,n}^e)$ is a torus. 
Now assume $m$ and $n$ are even and one is greater than four. Set $\gamma= \gcd(m,n)$. 
There are the following two cases.
\begin{enumerate}
\item If both $m/\gamma$ and $n/\gamma$ are odd, then $\mathit{genus}(X_{m,n}^e)=\frac{mn-m-n-2\gamma}{4}+1$, and $\omega^e_{m,n}$ has
$\gamma$ zeros each of order $\frac{mn-m-n}{2\gamma}-1$.
\item Otherwise, $\mathit{genus}(X_{m,n}^e)=\frac{mn-m-n-\gamma}{4}+1$, and the $1$-form $\omega^e_{m,n}$ has
$\gamma/2$ zeros each of order $\frac{mn-m-n}{\gamma}-1$.
\end{enumerate}
\end{theorem}

\subsection{Decomposition into semiregular polygons}
\label{sect:semiregular_background}

The {\em $(a,b)$-semiregular $2n$-gon} is the $2n$-gon whose edge vectors (oriented counterclockwise) are given by 
$${\bf v}_i=\begin{cases} 
a (\cos \frac{i\pi}{n}, \sin \frac{i\pi}{n}) & \textrm{if $i$ is even} \\
b (\cos \frac{i\pi}{n}, \sin \frac{i\pi}{n}) & \textrm{if $i$ is odd} 
\end{cases}
$$
for $i=0, \ldots, 2n-1$. Denote this $2n$-gon by $P_n(a,b)$. 
The edges whose edge vectors are ${\bf v}_i$ for $i$ even are called {\em even edges}. The remaining edges are called {\em odd edges}. 
We restrict to the cases where $a \geq 0$ and $b \geq 0$, but $a \neq 0$ or $b \neq 0$. In the case 
where one of $a$ or $b$ is zero, $P_n(a,b)$ degenerates to a regular $n$-gon. In the case where $a=0$ or $b=0$ and $n=2$, $P_n(a,b)$ degenerates to an edge.

Note that the exterior angles of a non-degenerate semiregular $2n$-gon are all equal to $\frac{\pi}{n}$. In addition, the polygon can be inscribed in a circle. In fact, all polygons which can be inscribed in a circle and have all equal angles are similar to either a regular polygon or a semiregular polygon. 

\begin{figure}[h]
\begin{center}
\includegraphics{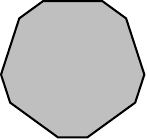}
\caption{The semiregular polygon $P_5(1,2).$}
\label{fig:semiregular}
\end{center}
\end{figure}

Fix $m$ and $n$. Define the polygons $P(k)$ for $k=0, \ldots, m-1$ by
\begin{equation}
P(k)=
\begin{cases}
P_n( \sin \frac{(k+1) \pi}{m}, \sin \frac{k \pi}{m}) & \textrm{if $n$ is odd} \\
P_n(\sin \frac{k \pi}{m}, \sin \frac{(k+1) \pi}{m}) & \textrm{if $n$ is even and $k$ is even} \\
P_n(\sin \frac{(k+1) \pi}{m},\sin \frac{k \pi}{m}) & \textrm{if $n$ is even and $k$ is odd.}
\end{cases}
\end{equation}
We form a surface by identifying the edges of the polygons in pairs. For $k$ odd, we identify the even sides of $P(k)$ with the opposite side
of $P(k+1)$, and identify the odd sides of $P(k)$ with the opposite side of $P(k-1)$. 
The cases in the definition of $P(k)$ are chosen so that this gluing makes sense.
We call the resulting surface $(Y_{m,n}, \eta_{m,n})$.
See examples in figures \ref{fig:sp64} and \ref{fig:sp45}.
Ronen Mukamel independently discovered these surfaces.

\begin{figure}[h]
\begin{center}
\includegraphics{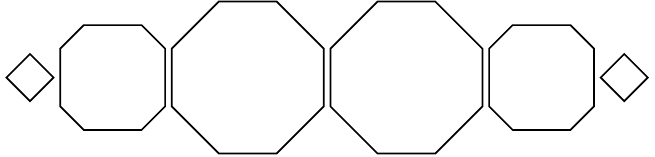}
\caption{These polygons make up one component of the surface $(Y_{6,4}, \eta_{6,4})$. These are the polygons $P(0)$, $P(1)$, $P(2)$, $P(3)$, $P(4)$ and $P(5)$ from left to right.}
\label{fig:sp64}
\end{center}
\end{figure}

\begin{figure}[h]
\begin{center}
\includegraphics{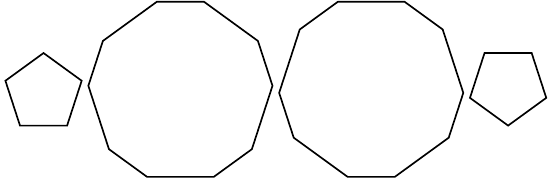}
\caption{These polygons make up one component of the surface $(Y_{4,5}, \eta_{4,5})$. These are the polygons $P(0)$, $P(1)$, $P(2)$ and $P(3)$ from left to right.}
\label{fig:sp45}
\end{center}
\end{figure}

\begin{theorem}[Semiregular decomposition]
\label{thm:semiregular}
There are affine homeomorphisms 
$$\mu:(X_{m,n}, \omega_{m,n}) \to (Y_{m,n}, \eta_{m,n}) \quad \text{and} \quad
\nu:(X_{m,n}, \omega_{m,n}) \to (Y_{n,m}, \eta_{n,m})$$ 
with derivatives
$$D(\mu)=\left[\begin{array}{rr}
\csc \frac{\pi}{n} & -\cot \frac{\pi}{n} \\
0 & 1 
\end{array}\right]
\quad \textrm{and} \quad
D(\nu)=\left[\begin{array}{rr}
-\csc \frac{\pi}{m} & -\cot \frac{\pi}{m} \\
0 & 1 
\end{array}\right].
$$
\end{theorem}

Of course, when $m$ and $n$ are even, the surface $(Y_{m,n}, \eta_{m,n})$ is not primitive. In this case, the polygons $P(i)$ and $P(m-1-i)$ differ only by translation.
In particular, there is an automorphism $\iota'_\ast:(Y_{m,n}, \eta_{m,n}) \to (Y_{m,n}, \eta_{m,n})$ which interchanges $P(i)$ and $P(m-1-i)$ for $i=0,1, \ldots, m/2-1$ and for which
$D(\iota'_\ast)=I$. Moreover, $\iota'_\ast=\mu \circ \iota_\ast \circ \mu^{-1}$. (See proposition \ref{prop:iota_ast}.)
We use $(Y^e_{m,n}, \eta^e_{m,n})$ to denote $(Y_{m,n}, \eta_{m,n})/\iota'_\ast$. Given this, the following is a corollary the theorem above.

\begin{corollary}
\label{cor:semiregular2}
For $m$ and $n$ even, there are affine homeomorphisms $\mu^e:(X_{m,n}^e, \omega_{m,n}^e) \to (Y^e_{m,n}, \eta^e_{m,n})$ and
$\nu^e:(X_{m,n}^e, \omega_{m,n}^e) \to (Y^e_{n,m}, \eta^e_{n,m})$, with $D(\mu^e)=D(\mu)$ and $D(\nu^e)=D(\nu)$.
\end{corollary}

We will now give explicit formulas for the Riemann surfaces and $1$-forms for the surfaces $(Y_{m,n}, \eta_{m,n})$ and $(Y_{m,n}^e, \eta_{m,n}^e)$. 
Compare the following to \cite[theorem 5.15]{BM}. (All but the third formula, which is closely connected to the second, appear in \cite{BM}.
However, the formulas are applied more broadly here.)

\begin{proposition}
\label{prop:formulas}
Assume $mn \geq 6$. We have the following formulas for the primitive Riemann surfaces and holomorphic differentials
$(Y_{m,n}, \eta_{m,n})$ and $(Y^e_{m,n}, \eta^e_{m,n})$
(up to scaling and rotating). 
\begin{enumerate}
\item If $m$ is odd, then $Y_{m,n}$ is defined by 
$\displaystyle y^{2n}=(u-2)\prod_{j=1}^{(m-1)/2} \Big(u-2 \cos \frac{2j\pi}{m}\Big)^2$,
and 
$\displaystyle \eta_{m,n}=\frac{y~\d u}{(u-2) \prod_{j=1}^{(m-1)/2} (u-2 \cos \frac{2j\pi}{m})}.$
\vspace{0.5em}

\item If $m$ is even and $n$ is odd, then 
$Y_{m,n}$ is defined by 
$\displaystyle y^{2n}=(u-2)^n\prod_{j=1}^{m/2} \Big(u-2 \cos \frac{(2j-1)\pi}{m}\Big)^2$,
and 
$\displaystyle \eta_{m,n}=\frac{y~\d u}{(u-2) \prod_{j=1}^{m/2} (u-2 \cos \frac{(2j-1)\pi}{m})}$.
\vspace{0.5em}

\item If both $m$ and $n$ are even, then $Y^e_{m,n}$ is defined by
$\displaystyle y^{n}=(u-2)^\frac{n}{2} \prod_{j=1}^{m/2} \Big(u-2 \cos \frac{(2j-1)\pi}{m}\Big)$,
and
$\displaystyle \eta^e_{m,n}=\frac{y~\d u}{(u-2) \prod_{j=1}^{m/2} (u-2 \cos \frac{(2j-1)\pi}{m})}$.
\end{enumerate}
\end{proposition}

\begin{corollary}
\label{cor:same}
When $m$ and $n$ are relatively prime, $(Y_{m,n}, \eta_{m,n})$ is the same as a surface constructed in \cite[theorem 5.15]{BM}.
\end{corollary}

When $\gcd(m,n) \neq 1$, it remains to describe the relationship between the surfaces with Veech groups commensurable to $\Delta(m,n,\infty)$ constructed in
this paper and those constructed by Bouw and M\"oller \cite{BM}. Work of Wright has
recently shown that these surfaces are always the same \cite{Wright12}. 

\subsection{Locations of proofs}
\label{sect:outline}

The author has strived to make the proofs of each major result above readable independently. The paper has been separated into the following sections.
The semiregular decomposition theorem is the main tool of the paper. We prove it in section \ref{sect:semiregular}.
We study the topology of these surfaces in section \ref{sect:topology}. 
In section \ref{sect:veech_group}, we compute the Veech groups of these surfaces.
We prove our primitivity results in section \ref{sect:primitivity}. Section \ref{sect:formulas} discusses our formulas for the Riemann surfaces and $1$-forms given in proposition \ref{prop:formulas}. Finally, in section \ref{sect:not_veech},
we discuss the proof of Theorem \ref{thm:non_veech_triangle_groups}
which states that some triangle groups are not Veech groups.

\begin{remark}
In this version of the paper, nearly all results are proved in terms of the semiregular decomposition, and the grid graph description of the surfaces is
just a bridge between the affinely equivalent surfaces $(Y_{m,n}, \eta_{m,n})$ and $(Y_{n,m}, \eta_{n,m})$. It is possible, though more cumbersome, to compute the Veech group
and topology of these surfaces through the grid graph description. This was the point of view of an earlier version of this paper \cite{HooperBouwMoller}.
\end{remark}

\section{The semiregular polygon decomposition}
\label{sect:semiregular}

\def\d{{\bf d}}

In this section we prove theorem \ref{thm:semiregular}, which provides a decomposition of the surface $(X_{m,n}, \omega_{m,n})$ into semiregular polygons,
up to an affine transformation. The theorem provides two such decompositions. We will first prove the existence of $\mu:(X_{m,n}, \omega_{m,n}) \to (Y_{m,n}, \eta_{m,n})$, which provides a 
decomposition of $(X_{m,n}, \omega_{m,n})$ into semiregular $2n$-gons, up to an affine transformation. This is the difficult part of the theorem. 
Then, we will analyze the subgroup of $\SL(2,\R)$ which preserves
the set of horizontal and vertical directions. This is a dihedral group of order $8$. We will see that $(X_{m,n}, \omega_{m,n})$ and $(X_{n,m}, \omega_{n,m})$ differ only by an element of
this dihedral group. In particular, the existence of $\mu$ will imply the existence of a $\nu:(X_{m,n}, \omega_{m,n}) \to (Y_{n,m}, \eta_{n,m})$.

\subsection{The existence of $\mu:(X_{m,n}, \omega_{m,n}) \to (Y_{m,n}, \eta_{m,n})$}
\label{sect:mu}
We begin by describing a decomposition of $(X_{m,n}, \omega_{m,n})$ into polygons. These will be the analogs of the polygons $P(0), \ldots, P(m-1)$ making up $(Y_{m,n}, \eta_{m,n})$. 

For ease of exposition, we consider the augmented graph $\G'_{m,n}$ obtained by attaching {\em degenerate nodes} and {\em degenerate edges}
to the graph $\G_{m,n}$. 
The nodes of $\G_{m,n}$ are in bijection with the
coordinates $(i,j) \in \Z^2$ with $0<i<m$ and $0<j<n$. The nodes of $\G_{m,n}'$ will be in bijection with those $(i,j) \in \Z^2$ with $0 \leq i \leq m$ and $0 \leq j \leq n$. 
Our added nodes are called {\em degenerate nodes}.
We join new {\em degenerate edges} between nodes of distance one in the plane that are not already joined by an edge. 
Our graph $\G_{m,n}'$ is also bipartite, and we follow the same naming conventions for nodes as when discussing $\G_{m,n}$. See equation \ref{eq:node_labels},
and the text below.
An example graph is shown in figure \ref{fig:degenerate_graph}.

Let $\E'$ denote the set of all edges of $\G'_{m,n}$.
We call a degenerate edge $e \in \E'$ {\em $\Alpha$-degenerate}, {\em $\Beta$-degenerate} or {\em completely degenerate} if
$\partial e$ contains a degenerate $\Alpha$-node, a degenerate $\Beta$-node or both, respectively.
We also define permutations $\e',\n':\E'\to \E'$ following convention \ref{permutation_convention}.

\begin{figure}[h]
\begin{center}
\includegraphics[height=3in]{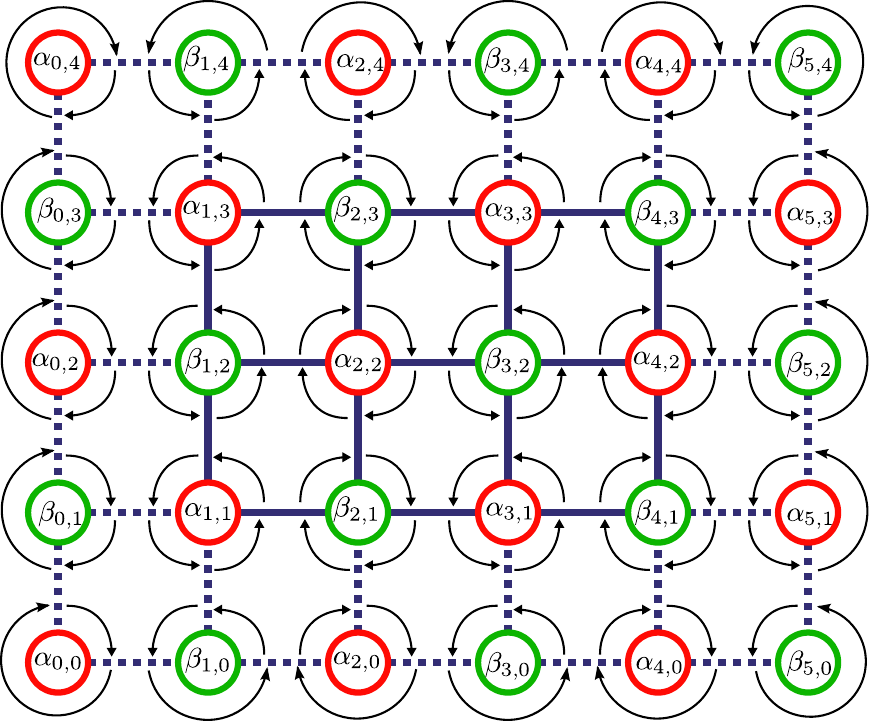}
\caption{The augmented graph $\G'_{5,4}$. The degenerate edges are drawn as dotted lines. The map $\e'$ is
given by the arrows surrounding the $\alpha$ vertices, and the map $\n'$ is given by the arrows surrounding the $\beta$ vertices.}
\label{fig:degenerate_graph}
\end{center}
\end{figure}

These degenerate edges correspond to degenerate rectangles on our surface $(X_{m,n}, \omega_{m,n})$. 
A {\em degenerate rectangle} is a rectangle with zero width or zero height. (The added nodes correspond to cylinders of zero width according to 
equation \ref{eq:eigenfunction}.)
The $\Alpha$-degenerate edges correspond to horizontal saddle connections (rectangles with zero height)
and the $\Beta$-degenerate edges correspond to vertical saddle connections. The completely degenerate edges correspond to points on our surface.

Each edge $e \in \E'$ corresponds to a rectangle (or degenerate rectangle) $R_e=R(e)$ in the surface $(X_{m,n}, \omega_{m,n})$ with horizontal and vertical sides. 
The {\em positive diagonal} of a rectangle with horizontal and vertical sides is the diagonal with positive slope. For a degenerate rectangle,
we take the positive diagonal to be the rectangle itself.
Let $\d(e)$ denote the vector which points along the positive diagonal, oriented rightward and upward.
The {lower triangle}, denoted $L(e)$, of a rectangle $R(e)$ is the triangle below the positive diagonal. The {upper triangle}, $U(e)$ is the triangle above the positive diagonal. See figure \ref{fig:rectangle}. For degenerate rectangles, we take $R(e)=L(e)=U(e)$ to be the corresponding saddle connection, or point.

\begin{figure}[h]
\begin{center}
\includegraphics[height=0.4in]{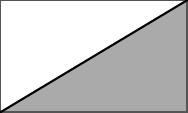}
\caption{A rectangle's positive diagonal. The lower triangle is shaded gray, and the upper triangle is white.}
\label{fig:rectangle}
\end{center}
\end{figure}

Recall that $\G'_{m,n}$ is naturally embedded in $\Z^2$. We use $v_{i,j}$ to denote the node of $\G'_{m,n}$ in the position $(i,j)$. 
We now define our decomposition of $(X_{m,n}, \omega_{m,n})$ into polygons. Let $H_k$ denote the set of edges of $G'_{m,n}$, 
\begin{equation}
\label{eq:Hk}
H_k=\{\overline{v_{k,i} v_{k+1,i}} \in \E' ~:~ 0 < i < n\} \quad 
\textrm{for $k=0,\ldots,m-1$}.
\end{equation}
($\bigcup_k H_k$ is the set of horizontal edges in the graph $\G'_{m,n}$, and the edges in each $H_k$ lie in a column.)
For each such $k$ define the polygon $Q(k) \subset (X_{m,n}, \omega_{m,n})$ by
\begin{equation}
\label{eq:qpolygon}
Q(k)=\bigcup_{e \in H_k} R(e) \cup L\big(\n'(e)\big) \cup L\big(\e'^{-1}(e)\big) \cup U\big(\n'^{-1}(e)\big) \cup U\big(\e'(e)\big).
\end{equation}
An example decomposition is shown in figure \ref{fig:stairs}.

\begin{figure}[h]
\begin{center}
\includegraphics[width=5.6in]{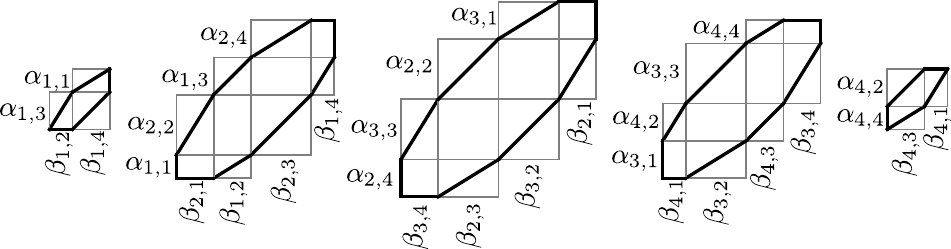}
\caption{The surface $(X_{5,5}, \omega_{5,5})$ decomposes into the polygons $Q(0)$, $Q(1)$, \ldots, $Q(4)$ ordered from left to right. Portions of the horizontal cylinders, $\alpha_\ast$, and the vertical cylinders, $\beta_\ast$ are labeled.} 
\label{fig:stairs}
\end{center}
\end{figure}

We have the following description of the affine homeomorphism $\mu:(X_{m,n}, \omega_{m,n}) \to (Y_{m,n}, \eta_{m,n})$. This implies half of theorem
\ref{thm:semiregular}.

\begin{lemma}
\label{lem:mu}
There is an affine homeomorphism $\mu:(X_{m,n}, \omega_{m,n}) \to (Y_{m,n}, \eta_{m,n})$ such that $\mu\big(Q(k)\big)=P(k)$ for $k=0, \ldots, m-1$.
Moreover, $D(\mu)=\left[\begin{array}{rr}
\csc \frac{\pi}{n} & -\cot \frac{\pi}{n} \\
0 & 1 
\end{array}\right]$.
\end{lemma}
\begin{proof}
Let $M$ denote the matrix identified as $D(\mu)$ in the lemma. 
The proof consists of two parts.  First we show that
$M\big(Q(k)\big)$ is the same as $P(k)$ up to translation. Second, we show that the boundary edges of the polygons $Q(k)$ considered
as subsets of $(X_{m,n}, \omega_{m,n})$ are identified in the same combinatorial way as the polygons $P(k)$ which make up $(Y_{m,n}, \eta_{m,n})$. 
Concretely, we are defining the affine homeomorphism on pieces, and checking that the homeomorphisms agree on the boundaries. 
If this is true, then the homeomorphisms extend to the whole surface.

We will show that these subsets $Q(k)$ are in fact polygons, and $M\big(Q(k)\big)=P(k)$. We break into two cases depending on the parity of $k$. 

Fix an odd integer $k$ satisfying $0<k \leq m-1$. 
Define the edge $e_i=\overline{v_{k,i} v_{k+1,i}}$ for $i=0, \ldots, n$. We have $e_i \in H_k$ when $i=1,\ldots, n-1$. 
We have that $$\n' \circ \e'(e_i)=\e' \circ \n' (e_i)=e_{i+1}.$$ Therefore many of the triangles are mentioned twice in 
equation \ref{eq:qpolygon}. (For instance, $L\big(\n' (e_1)\big)=L\big(\e'^{-1} (e_2)\big)$.)
Moreover, the top right coordinate vertex $R(e_i)$ is the same as the bottom left vertex of $R(e_{i+1})$ and this point is non-singular
provided neither rectangle is degenerate. Thus this point is non-singular for $i=1, \ldots, n-2$. 
With this in mind, we see that $Q(k)$ is formed by a chain of rectangles $R(e_i)$ moving to the northeast with some triangles added on.
In particular, for $k=1, \ldots, m-2$, $Q(k)$ has $2n$ sides. When $k=m-1$, half of these sides will degenerate to points.
We treat these cases as $2n$-gons as well, with half of their edges of length $0$. 
Let ${\bf u}_i$ for $i=0, \ldots, 2n-1$ be the edge vectors of $Q(k)$ oriented counterclockwise around $Q(k)$. We assume the first edge
vector ${\bf u}_0$ is the lower horizontal edge of the rectangle $R(e_1)$. 
(We have  ${\bf u}_0=\d(\n'^{-1}(e_1))=\d(\overline{\alpha_{k+1,0} \beta_{k+1,1}})$.)
We find these edge vectors of $Q(k)$ to be 
\begin{equation}
\label{eq:wi_odd}
{\bf u}_i=\begin{cases}
\d(\overline{\alpha_{k+1,i} \beta_{k+1,i+1}}) & \textrm{if $i<n$ and $i$ even} \\
\d(\overline{\alpha_{k,i} \beta_{k,i+1}}) & \textrm{if $i<n$ and $i$ odd} \\
-\d(\overline{\beta_{k+1,2n-1-i} \alpha_{k+1,2n-i}}) & \textrm{if $i \geq n$ and $i$ even} \\
-\d(\overline{\beta_{k,2n-1-i} \alpha_{k,2n-i}}) & \textrm{if $i \geq n$ and $i$ odd.}
\end{cases}
\end{equation}
Therefore, we have
$${\bf u}_i=\begin{cases}
\sin \frac{(k+1)\pi}{m}(\sin\frac{(i+1) \pi}{n},\sin \frac{i \pi}{n}) & \textrm{if $i$ is even} \\
\sin \frac{k\pi}{m}(\sin\frac{(i+1) \pi}{n},\sin \frac{i \pi}{n}) & \textrm{if $i$ is odd.}
\end{cases}$$
By a simple trigonometric calculation, 
$$M {\bf u}_i=\begin{cases}
\sin\frac{(k+1) \pi}{m} ( \cos \frac{i \pi}{n}, \sin \frac{i\pi}{n}) & \textrm{if $i$ is even} \\
\sin\frac{k \pi}{m} ( \cos \frac{i \pi}{n}, \sin \frac{i\pi}{n}) & \textrm{if $i$ is odd}.
\end{cases}$$
Thus, $M(Q(k))=P_n(\sin \frac{(k+1) \pi}{m}, \sin\frac{k \pi}{m})$, the same polygon as $P(k)$.

The case of $k$ even with $0 \leq k \leq m-1$ is similar. Let $e_i=\overline{v_{k,i} v_{k+1,i}}$ for $i=0, \ldots, n$. We have 
$\n'^{-1} \circ \e'^{-1}(e_i)=\e'^{-1} \circ \n'^{-1} (e_i)=e_{i+1}$.
So, again the lower left and top right vertices are non-singular. But, the chain of rectangles $R(e_i)$ moves toward the southwest.
Again, it can be observed that $Q(k)$ is a $2n$-gon, which is degenerate if $k=0$ or $k=m-1$. 
We would like to compute the edge vectors ${\bf w}_i$ for $i=0, \ldots, 2n-1$. 
We set ${\bf w}_0$ to be the lower horizontal edge of $Q(k)$. We see
${\bf w}_0=\d\big(\e(e_n)\big)$. Thus, we introduce the variable $j$ defined by $i=n+j$. 
The edge vectors are defined as follows.
\begin{equation}
\label{eq:wi_even}
{\bf w}_{i}={\bf w}_{n+j}=\begin{cases}
-\d(\overline{\alpha_{k,j} \beta_{k,j+1}}) & \textrm{if $j \geq 0$ and $j$ even} \\
-\d(\overline{\alpha_{k+1,j} \beta_{k+1,j+1}}) & \textrm{if $j \geq 0$ and $j$ odd} \\
\d(\overline{\beta_{k,-j-1} \alpha_{k,-j}}) & \textrm{if $j < 0$ and $j$ even} \\
\d(\overline{\beta_{k+1,-j-1} \alpha_{k+1,-j}}) & \textrm{if $j < 0$ and $j$ odd.}
\end{cases}
\end{equation}
We see 
$${\bf w}_{i}={\bf w}_{n+j}=\begin{cases}
-\sin \frac{k\pi}{m}(\sin\frac{(j+1) \pi}{n},\sin \frac{j \pi}{n}) & \textrm{if $j$ is even} \\
-\sin \frac{(k+1)\pi}{m}(\sin\frac{(j+1) \pi}{n},\sin \frac{j \pi}{n}) & \textrm{if $j$ is odd.}
\end{cases}$$
We have
$$M{\bf w}_{i}=M{\bf w}_{n+j}=\begin{cases}
\sin \frac{k\pi}{m}(\cos \frac{i \pi}{n}, \sin \frac{i\pi}{n}) & \textrm{if $i-n=j$ is even} \\
\sin \frac{(k+1)\pi}{m}(\cos \frac{i \pi}{n}, \sin \frac{i\pi}{n}) & \textrm{if $i-n=j$ is odd}.
\end{cases}$$
Thus, $M(Q(k))=P_n(\sin \frac{k\pi}{m},\sin \frac{(k+1)\pi}{m})$ when $n$ is even and $M(Q(k))=P_n(\sin \frac{(k+1)\pi}{m},\sin \frac{k\pi}{m})$ when $n$ is odd.
In either case, we have $M(Q(k))=P(k)$. 

Finally, we note that the identification of edges of these polygons agrees with the gluing definition given in section \ref{sect:semiregular_background}.
Fix $k$ odd. We will see that the even sides of $Q(k)$ are identified with the opposite sides of $Q(k+1)$, and the odd sides of $Q(k)$ are identified with the opposite
sides of $Q(k-1)$. This needs to be done in four cases, of which we will only do one.

Fix an even integer $i<n$, 
Since $k$ is odd, up to sign, the edge of $Q(k)$ associated to vector ${\bf u}_i$ is the positive diagonal of the rectangle $R(\overline{\alpha_{k+1,i} \beta_{k+1,i+1}})$.
This edge of $Q(k)$ is also the $n+i$-th edge of $Q(k+1)$, whose edge vector is $w_{n+i}=-u_i$. (Note that $k$ must be replaced with $k+1$ in the formula for $w_{n+i}$ to see this because we are looking at $Q(k+1)$.) This vector $w_{n+i}$ also comes from the positive diagonal
of $R(\overline{\alpha_{k+1,i} \beta_{k+1,i+1}})$. Thus,
these two edges are identified in $(X_{m,n}, \omega_{m,n})$. 

The remaining three cases to cover are when $i$ is even and $i \geq n$, $i$ is odd and $i<n$, and when $i$ is odd and $i \geq n$. These
cases are left to the motivated reader.
\end{proof}

We have the following corollary about the action of the automorphism $\iota_\ast:(X_{m,n}, \omega_{m,n}) \to (X_{m,n}, \omega_{m,n})$, which exists when $m$ and $n$ are even.
Recall, $\iota_\ast$ was induced by a graph automorphism $\iota: \G_{m,n} \to  \G_{m,n}$. See equation \ref{eq:iota}
of \S \ref{sect:bm}.

\begin{proposition}
\label{prop:iota_ast}
Suppose $m$ and $n$ are even. Then $\iota_\ast\big(Q(k)\big)=Q(m-1-k)$ for $k=0, \ldots m-1$. It follows that 
$\iota'_\ast=\mu \circ \iota_\ast \circ \mu^{-1}:(Y_{m,n}, \eta_{m,n}) \to (Y_{m,n}, \eta_{m,n})$ is an affine automorphism
of $(Y_{m,n}, \eta_{m,n})$ with $D(\iota'_\ast)=I$ and $\iota'_\ast\big(P(k)\big)=P(m-1-k)$ for $k=0, \ldots m-1$.
\end{proposition}
\begin{proof}
Recall the definition of $H_k \subset \E'$ given in equation \ref{eq:Hk}.
Note that the graph automorphism $\iota$ extends naturally to $\G_{m,n}'$ and 
satisfies $\iota(H_k)=H_{m-1-k}$. By the definition of $Q(k)$ in equation \ref{eq:qpolygon},
$\iota_\ast\big(Q(k)\big)=Q(m-1-k)$.
\end{proof}

\subsection{The dihedral group}
\label{sect:dihedral}

The dihedral group of order 8, $D_8$, acts on the plane in a way that preserves the set of directions $\{\textrm{horizontal}, \textrm{vertical}\}$. In particular,
if $S$ is a translation surface with horizontal and vertical cylinder decompositions and if $M \in D_8$, then the natural affine homeomorphism $S \to M(S)$ preserves the collection of
all horizontal and vertical cylinders. Thus, there is an action of $D_8$ on the data associated to the cylinder intersection graph. 
Note that the matrices $C$ and $E$ given in equation \ref{eq:matrices} generate $D_8$. We will record their actions on this data.

\begin{proposition}[Action of $D_8$]
\label{prop:dihedral}
Let $S[\G,(\Alpha,\Beta),(\e,\n),w]$ denote the translation surface constructed from the bipartite ribbon graph $\G$ with vertex
set $\V=\Alpha \cup \Beta$, width function $w:\V \to \R_{>0}$, edge set $\E$, and edge permutations $\e, \n:\E \to \E$, as in \S \ref{sect:thurston_veech}.
Then
\begin{itemize}
\item $C\big(S[\G,(\Alpha,\Beta),(\e,\n),w]\big)=S[\G,(\Beta,\Alpha),(\n^{-1},\e^{-1}),w]$, and
\item $E\big(S[\G,(\Alpha,\Beta),(\e,\n),w]\big)=S[\G,(\Alpha,\Beta),(\e^{-1},\n),w]$.
\end{itemize}
\end{proposition}

In our setting, this gives us the following. 

\begin{corollary}
\label{cor:e}
There is an affine homeomorphism $\rho:(X_{m,n}, \omega_{m,n}) \to (X_{n,m}, \omega_{n,m})$ with derivative $E$. 
In particular, $E \in \GL(X_{m,m}, \omega_{m,m})$ for all $m$.
\end{corollary}
\begin{proof}
The graph homomorphism $\eta :\G_{m,n} \to \G_{n,m}: v_{i,j} \mapsto v_{j,i}$ satisfies 
$\eta(\Alpha_{m,n})=\Alpha_{n,m}$, $\eta(\Beta_{m,n})=\Beta_{n,m}$, $\eta \circ \e_{m,n} \circ \eta^{-1}=\e_{n,m}^{-1}$, $\eta \circ \n_{m,n} \circ \eta^{-1}=\n_{n,m}$,
and $w_{m,n} \circ \eta=w_{n,m}$.
\end{proof}

This corollary was the last piece we needed to prove the semiregular decomposition, theorem \ref{thm:semiregular}.

\begin{proof}[Proof of theorem \ref{thm:semiregular}.]
Lemma \ref{lem:mu} handles the existence of $\mu$, so we will concentrate on the existence of $\nu$. 
Let $\mu'$ denote the affine homeomorphism $(X_{n,m}, \omega_{n,m}) \to (Y_{n,m}, \eta_{n,m})$ guaranteed by lemma \ref{lem:mu}.
We define the affine homeomorphism $\nu:(X_{m,n}, \omega_{m,n}) \to (Y_{n,m}, \eta_{n,m})$ to be $\nu=\mu' \circ \rho$, with $\rho:(X_{m,n}, \omega_{m,n}) \to (X_{n,m}, \omega_{n,m})$ as in corollary \ref{cor:e}. 
We have that 
$$D(\nu)=D(\mu') \cdot D(\rho)=\left[\begin{array}{rr}
-\csc \frac{\pi}{m} & -\cot \frac{\pi}{m} \\
0 & 1 
\end{array}\right].$$
\end{proof}

\section{Topological type}
\label{sect:topology}
\def\T{\mathcal T}

In this section, we will compute the topological types of the surfaces $(Y_{m,n}, \eta_{m,n})$ and $(Y_{m,n}^e, \eta_{m,n}^e)$. 
Recall $(Y_{m,n}, \eta_{m,n})$ decomposes into semiregular polygons $P(0), \ldots, P(m-1)$.

\begin{proposition}[Singularities of $(Y_{m,n}, \eta_{m,n})$]
\label{prop:sing1}
Let $\gamma=\gcd(m,n)$. There are $\gamma$ equivalence classes of vertices of the decomposition into polygons. 
Each of these points has cone angle $2 \pi (mn-m-n)/\gamma$. 
\end{proposition}
\begin{proof}
Let ${\mathbf v}_0$ be a vector based at a vertex $v$ of $P(0)$, which is pointing along the boundary of $P(0)$ in the counterclockwise 
direction. We will rotate this vector counterclockwise around the point $V$ of $(Y_{m,n}, \eta_{m,n})$ which is the equivalence class of vertices of polygons containing $v$.
$P(0)$ is a regular $n$-gon, so we reach $P(1)$ when we have rotated by 
$\pi-\frac{2\pi}{n}$. Inside $P(1)$ we may rotate by another $\pi-\frac{2 \pi}{2n}$ until we reach $P(2)$, since
$P(1)$ is a semiregular $2n$-gon. For $i=1, \ldots, m-2$, $P(i)$ is a regular $2n$-gon. Thus we repeat this process until we reach
$P(m-1)$. Then $P(m-1)$ is a regular $n$-gon again, so we rotate by $\pi-\frac{2\pi}{n}$. Now the indices decrease. When we rotate
by $(m-2)(\pi-\frac{2 \pi}{2n})$ we reach $P(0)$. We have closed up if the total rotation we have done is a multiple of $2 \pi$. In general, we see that the 
cone angle at $V$ is 
$$x \big(2(\pi-\frac{2\pi}{n})+2(m-2)(\pi-\frac{2 \pi}{2n})\big)=2x\pi\frac{nm-n-m}{n},$$
where $x$ is the smallest positive integer for which this number is a multiple of $2 \pi$. We see $x=n/\gamma$.
So $V$ has cone angle $2 \pi (mn-m-n)/\gamma$. 
Every singularity is of this form, and the sum of all the angles of polygons $P(0), \ldots, P(m-1)$ is $2 \pi(mn-m-n)$.
Hence, there are $\gamma$ total singularities.
\end{proof}

\begin{corollary}
The Euler characteristic of $(Y_{m,n}, \eta_{m,n})$ is $m+n+\gamma-mn$. 
\end{corollary}
\begin{proof}
The decomposition $(Y_{m,n}, \eta_{m,n})=\bigcup_{i=0}^{m-1} P(i)$ has
$m$ faces, $(m-1)n$ edges, and $\gamma$ vertices.
\end{proof}

\begin{proposition}[Singularities of $(Y_{m,n}^e, \eta_{m,n}^e)$]
If $m/\gamma$ and $n/\gamma$ are odd, then there are $\gamma$ equivalence classes of vertices in the polygonal decomposition of $(Y_{m,n}^e, \eta_{m,n}^e)$
and each has cone angle $\pi (mn-m-n)/\gamma$.
Otherwise, when one of $m/\gamma$ or $n/\gamma$ is even, there are $\gamma/2$ equivalence classes of vertices and each has cone angle 
$2\pi (mn-m-n)/\gamma$.
\end{proposition}
\begin{proof}
The proof proceeds in the same manner. Recall that
$(Y_{m,n}^e, \eta_{m,n}^e)$ is the quotient $(Y_{m,n}, \eta_{m,n})/\iota'_\ast$,
where $\iota'_\ast$ is the order two action which switches the polygons
$P(i)$ with $P(m-1-i)$ for all $i$. 
Let ${\mathbf v}_0$ be a vector based at a vertex $v$ of the polygon $P(0)/\iota'_\ast$
in $(Y_{m,n}^e, \eta_{m,n}^e)$ which is pointing along the boundary of $P(0)$ in a counterclockwise 
direction. We rotate ${\mathbf v}_0$ counterclockwise. We continue rotating until we get back to $P(0)/\iota'_\ast=P(m-1)/\iota'_\ast$. At this point, we have rotated by
$\pi\frac{nm-n-m}{n}$. The cone angle at this point will be
$$x\pi\frac{nm-n-m}{n},$$
where $x$ is the smallest positive integer which makes this an integer multiple of $2\pi$. Thus $x=2n/\gcd(2n,m+n)$. 
We have 
$$\gcd(2n,m+n)=\gamma \gcd(\frac{2n}{\gamma},\frac{m}{\gamma}+\frac{n}{\gamma})=
\begin{cases}
2 \gamma & \textrm{if both $\frac{m}{\gamma}$ and $\frac{n}{\gamma}$ are odd} \\
\gamma & \textrm{if one of $\frac{m}{\gamma}$ or $\frac{n}{\gamma}$ is even.} 
\end{cases}
$$
This determines the cone angle. The number of vertices follows by dividing the total angle
by this cone angle.
\end{proof}

We can compute the Euler characteristic as before.

\begin{corollary}
If both $m/\gamma$ and $n/\gamma$ are odd, then $\chi(X_{m,n}^e, \omega_{m,n}^e)=\frac{m+n+2\gamma-mn}{2}$.
Otherwise, $\chi(X_{m,n}^e, \omega_{m,n}^e)=\frac{m+n+\gamma-mn}{2}$.
\end{corollary}

\section{The Veech groups}
\label{sect:veech_group}

\subsection{The orthogonal groups of $(Y_{m,n}, \eta_{m,n})$}

Given a translation surface $(X, \omega)$, we define the {\em orthogonal group} $O(X, \omega)=\GL(X, \omega) \cap O(2, \R)$. 
These are the derivatives of affine automorphisms which preserve the Euclidean metric. We define the additional matrix
\begin{equation}
\label{eq:orthogongal_matrices}
Y_n=\left[\begin{array}{rr} 
\cos \frac{\pi}{n} & -\sin \frac{\pi}{n} \\
-\sin \frac{\pi}{n} & -\cos \frac{\pi}{n}
\end{array}\right]
\end{equation}
The matrices $E$ and $Y_n$ generate a dihedral group of order $4n$, and satisfy the relations
$E^2=Y_n^2=I$ and $(EY_n)^n=-I$.

\begin{proposition}[Orthogonal group of $(Y_{m,n}, \eta_{m,n})$.]
\label{prop:orthogonal}
Suppose $(Y_{m,n}, \eta_{m,n})$ is not a torus.
If $m$ and $n$ are not both even, then $O(Y_{m,n}, \eta_{m,n})=\langle E, Y_n \rangle$.
If both $m$ and $n$ are even, then $O(Y_{m,n}, \eta_{m,n})=\langle E, Y_n E Y_n \rangle$, a dihedral group of order $2n$.
\end{proposition}
\begin{proof}
Recall, $(Y_{m,n}, \eta_{m,n})$ is a union of the semiregular $2n$-gons $P(0), P(1), \ldots, P(m-1)$. 
Both $E$ and $Y_nEY_n$ are symmetries of every semiregular $2n$-gon. In particular, there are affine automorphisms
of $(Y_{m,n}, \eta_{m,n})$ with derivatives $E$ and $Y_nEY_n$ which preserve the each of the polygons $P(0), P(1), \ldots, P(m-1)$. 
In addition, when $m$ or $n$ is odd, then $Y_n\big(P(i)\big)=P(m-1-i)$, up to translation. This action extends to an
affine automorphism of $(X_{m,n}, \omega_{m,n})$ with derivative $Y_n$.

Conversely, suppose $M \in O(Y_{m,n}, \eta_{m,n})$. Then the associated affine automorphism must permute the shortest saddle connections. These are the boundaries of the polygons
$P(0)$ and $P(m-1)$. (Since $(Y_{m,n}, \eta_{m,n})$ is not a torus, all the vertices are singularities by \ref{prop:sing1}.)
In particular, $M$ must preserve the set of directions in which these shortest saddle connections point. 
When $m$ and $n$ are not both even, the group of matrices with this property is $\langle E, Y_n \rangle$. When 
both $m$ and $n$ are even, $\langle E, Y_n E Y_n \rangle$ is the group of matrices with this property.
Consequently, $M$ must be in this group.
\end{proof}

We also cover the case of $(Y_{m,n}^e, \eta_{m,n}^e)$. The proof is nearly identical, so we omit it.

\begin{proposition}[Orthogonal group of $(Y_{m,n}^e, \eta_{m,n}^e)$.]
\label{prop:orthogonal2}
Suppose $m$ and $n$ are even, and $(Y_{m,n}^e, \eta_{m,n}^e)$ is not a torus.
Then $O(Y_{m,n}, \eta_{m,n})=\langle E, Y_n E Y_n \rangle$.
\end{proposition}

The proof proceeds in the same manner. 

\subsection{The Veech groups}

In this section, we prove theorems \ref{thm:veech_groups} and \ref{thm:even_veech_groups} which prescribe the Veech groups
of the surfaces $(X_{m,n}, \omega_{m,n})$ and $(X_{m,n}^e, \omega_{m,n}^e)$. 
The following establishes the Veech group of $(X_{m,n}, \omega_{m,n})$.

\begin{proof}[Proof of theorem \ref{thm:veech_groups}]
We leave it to the reader to check that when $mn<10$, then $(X_{m,n}, \omega_{m,n})$ is a torus. Now assume $mn \geq 10$. 
We have the following relations between matrices.
\begin{eqnarray*}
A=D(\nu)^{-1} \circ E \circ D(\nu), \quad
B=D(\mu)^{-1} \circ E \circ D(\mu) \quad \textrm{and} \\
C=D(\mu)^{-1} \circ Y_n \circ D(\mu)=-I \circ D(\nu)^{-1} \circ Y_m \circ D(\nu).
\end{eqnarray*}
In particular, theorem \ref{thm:semiregular} implies that $A,B \in \GL(X_{m,n}, \omega_{m,n})$, by pulling back the automorphisms. Similarly,
$C \in \GL(X_{m,n}, \omega_{m,n})$ when $m$ and $n$ not both even, and $CAC, CBC \in \GL(X_{m,n}, \omega_{m,n})$ when both $m$ and $n$ are even.
When $m=n$, we have $E \in \GL(X_{m,m}, \omega_{m,m})$ by corollary \ref{cor:e}. For all $m$ and $n$, this proves that the group described in 
theorem \ref{thm:veech_groups} is really contained in the $\GL(X_{m,n}, \omega_{m,n})$.

Now we will see that this is the whole Veech group. Let $\Gamma_{m,n} \subset \SL(2,\R)$ denote the orientation preserving subgroup of
the group described in theorem \ref{thm:veech_groups} as the Veech group of $(X_{m,n}, \omega_{m,n})$. We have shown above that 
$\Gamma_{m,n} \subset \SL(X_{m,n}, \omega_{m,n})$. 

Let $M$ be an orbifold which is topologically a $2$-sphere (possibly with punctures). Let
$\Sigma$ be the set of singularities of $M$. That is, $\Sigma$ is the collection of cone points and punctures of $M$.
In this specific case, the {\em Euler number} of $M$ is given by the formula 
\begin{equation}
\label{eq:euler}
\chi(M)=2+\sum_{\sigma \in \Sigma} (\frac{1}{|G_\sigma|}-1), 
\end{equation}
where $G_\sigma$ is the group associated to the singularity $s$, and $|G_\sigma|$ denotes the order of this group. 
Treat $1/|G_\sigma|=0$ if $G_\sigma$ is infinite, ie. when $\sigma$ is a puncture. 
For more information on the Euler number of an orbifold see \cite[chapter 13]{ThNotes}.
Note that a hyperbolic orbifold must have negative Euler number. Moreover, if $M \to N$ is a covering map of
degree $d$, then $\chi(N)=\chi(M)/d$. In particular, we have $\chi(M) \leq \chi(N)$ with equality implying that $M=N$.
Note further that adding more singular points only lowers the Euler number.

We apply this argument to the case $M={\mathbb H}^2/\Gamma_{m,n}$ and $N={\mathbb H}^2/\SL(X_{m,n}, \omega_{m,n})$.
We know both $M$ and $N$ are spheres, because $M$ is a sphere and covers $N$. 
We know that $N$ has at least one puncture, corresponding to the horizontal cylinder decomposition. If both $n$ and $m$ are even,
$N$ must have another puncture corresponding to the vertical cylinder decomposition. Here, no element of $\SL(X_{m,n}, \omega_{m,n})$ may send the horizontal
direction to the vertical direction. This is because when $m$ and $n$ are even, the number of maximal horizontal and vertical cylinders of $(X_{m,n}, \omega_{m,n})$ differ by one.

Now we consider the finite order singularities. These are fixed points of maximal orthogonal subgroups (orthogonal up to conjugation). We utilize theorem \ref{thm:semiregular} and proposition \ref{prop:orthogonal} to determine two subgroups of $\SL(X_{m,n}, \omega_{m,n})$ which are orthogonal. For two subgroups to count in the formula \ref{eq:euler}, they must not differ by conjugation in $\SL(X_{m,n}, \omega_{m,n})$. In particular, we use the fact that if they differ in orders, then they do not differ by conjugation.
We have the following special cases of orthogonal
subgroups, for which we can verify distinctness up to conjugation.
Note that for the orbifold calculation, we consider the group order in $\PSL(2,\R)$, because $-I$ acts trivially on $\H^2$.
\begin{itemize}
\item If $m$ and $n$ are not both even and $m \neq n$, we have groups of order $m$ and $n$.
\item If $m$ and $n$ are even with $m \neq n$, we have groups of order $m/2$ and $n/2$. 
\item If $m=n$ is odd, we have at least one group of order $m$ and one group of order a multiple of $2$. (The group of order two is
$\langle EC\rangle$. $EC$ can not be conjugated to lie in the group of order $m$ because $2$ does not divide $m$.)
\item If $m=n$ is even, we have at least one group of order $m/2$.
\end{itemize}
In all cases, we have determined that $\chi(M) \geq \chi(N)$ and thus $\SL(X_{m,n}, \omega_{m,n})=\Gamma_{m,n}$. 

Finally, we consider orientation reversing elements. To see that $\GL(X_{m,n}, \omega_{m,n})$ is as stated in the theorem, note that $A \in \GL(X_{m,n}, \omega_{m,n})$. The orientation 
preserving subgroup is always index two inside a group with orientation reversing elements. Thus a single orientation reversing
element plus the orientation preserving subgroup determine the whole group.
\end{proof}

The following establishes the Veech group of $(X_{m,n}^e, \omega_{m,n}^e)$ for $m$ and $n$ even.

\begin{proof}[Proof of theorem \ref{thm:even_veech_groups}]
Again, by proposition \ref{prop:orthogonal2} and corollary \ref{cor:semiregular2},
we see $A,B,CAC,CBC \in \GL(X_{m,n}^e, \omega_{m,n}^e)$ by pulling back the actions of the dihedral groups. Thus
$\GL(X_{m,n}, \omega_{m,n}) \subset \GL(X_{m,n}^e, \omega_{m,n}^e)$.

Now we check that this is everything. To do this, note that there must be two cusps in $\H^2/\SL(X_{m,n}^e, \omega_{m,n}^e)$, because there are again a different number
of horizontal and vertical cylinders. In addition for $m \neq n$, the maximal orthogonal subgroups (orthogonal up to conjugation) 
are of orders $m/2$ and $n/2$. So, $\H^2/\SL(X_{m,n}^e, \omega_{m,n}^e)$
has two cone points of these orders. When $m=n$, then we have at least one orthogonal group of order $m/2$. Then,
an orbifold Euler number computation shows that it must be that $\SL(X_{m,n}^e, \omega_{m,n}^e) = \SL(X_{m,n}, \omega_{m,n})$.
Again, we can see $\GL(X_{m,n}^e, \omega_{m,n}^e) = \GL(X_{m,n}, \omega_{m,n})$ by noting that $A$ appears in both groups.
\end{proof}

\section{Primitivity}
\label{sect:primitivity}

To show primitivity, we consider the following consequence of the theorem \ref{thm:primitive} of M\"oller.
Recall, if a translation surface has the lattice property, then it only covers a torus if its Veech group is arithmetic.
Let $(X, \omega)$ and $(X_0, \omega_0)$ be translation surfaces and $f:X \to X_0$ be a covering.
Following \cite{HS01}, we say $f$ is a {\em balanced covering} if the image of every zero of $\omega$ is a zero of
$\omega_0$. 

\begin{proposition}
Let $(X,\omega)$ be a translation surface which does not cover a torus. Let $f:X \to X_0$ be
the unique covering of a primitive translation surface $(X_0,\omega_0)$ guaranteed by theorem \ref{thm:primitive}. 
If $\Aff(X,\omega)$ acts transitively on the zeros of $\omega$, then $f$ is a {\em balanced covering}.  
\end{proposition}
\begin{proof}
We know $X_0$ is not a torus. Then $\omega_0$ has a zero $z \in X_0$. Choose $x \in X$ so that $f(x)=z$. 
Such an $x$ must be a zero of $\omega$. Let $y \in X$ be another zero. 
Then there is a $\rho \in \Aff(X,\omega)$ for which $\rho(x)=y$. 
Let $\phi:X_0 \to X_0$ be the affine automorphism with $D(\phi)=D(\rho)^{-1}$, which exists by theorem \ref{thm:primitive}.
This theorem also guarantees the uniqueness of the covering $X \to X_0$. Thus, $\phi \circ f \circ \rho=f$.
In particular, $f(y)=\phi^{-1} \circ f(x)=\phi^{-1}(z)$ which must be a zero.  
\end{proof}

In order to apply this, we need the following.

\begin{proposition}
$\Aff(X_{m,n}, \omega_{m,n})$ acts transitively on zeros of $\omega_{m,n}$. 
For $m$ and $n$ even, $\Aff(X^e_{m,n}, \omega^e_{m,n})$ acts transitively on zeros of $\omega^e_{m,n}$. 
\end{proposition}
\begin{proof}
We use the semiregular decomposition given by theorem \ref{thm:semiregular}. Consider the surface $(Y_{m,n}, \eta_{m,n})$,
which is a union of the semiregular $2n$-gons $P(0), P(1), \ldots, P(m-1)$.  Let $v \in (Y_{m,n}, \eta_{m,n})$ be a vertex of $P(k)$. Then
one of the adjacent edges to $v$ joins $P(k)$ to $P(k-1)$. Thus, $v$ is also a vertex of $P(k-1)$. By induction, we see that
$v$ is a vertex of $P(0)$. But, the group generated by a rotation by $\frac{2\pi}{n}$ preserves $P(0)$ setwise and acts
transitively on the vertices of $P(0)$. This group action extends to a group of affine automorphisms of $(Y_{m,n}, \eta_{m,n})$. 
Note that essentially the same argument works for $(X^e_{m,n}, \omega^e_{m,n})$.
\end{proof}

We have the following proof of primitivity. 

\begin{proof}[Proof of theorem \ref{thm:primitivity}]
We only consider the case where $m$ and $n$ are not both even. A slight variant of the argument below also holds for $(X^e_{m,n}, \omega^e_{m,n})$. 

Primitivity of $(Y_{m,n}, \eta_{m,n})$ is equivalent to primitivity of $(X_{m,n}, \omega_{m,n})$ by theorem \ref{thm:semiregular}.

Let $f:Y_{m,n} \to X_0$ be a covering of a primitive translation surface $(X_0,\omega_0)$. 
We know that $X_0$ is not a torus, since $(Y_{m,n}, \eta_{m,n})$ has the lattice property but $\SL(Y_{m,n}, \eta_{m,n})$ is not arithmetic.

By the propositions above, $f$ sends zeros of $\eta_{m,n}$ to zeros of $\omega_0$. 
Then it sends saddle connections to saddle connections. 
Thus, if $P$ is a convex polygon in $(Y_{m,n}, \eta_{m,n})$ whose boundary edges
are all saddle connections, then $f(P)$ also has this property in $(X_0, \omega_0)$. Suppose that $f$ is generically $k$-to-one.
Then given any such $P$, the collection of polygons in $f^{-1} \circ f(P)$ is a collection of $k$ isometric polygons with disjoint interiors that are bounded by saddle connections and differ only by translation.
Now consider the polygon $P(0) \subset Y_{m,n}$. The saddle connections bounding $P(0)$ are the shortest of all the saddle connections
of $(Y_{m,n}, \eta_{m,n})$, and the only other saddle connections that are this short bound $P(m-1)$. But, $P(m-1)$ is not a translate of $P(0)$. 
So $f$ is generically one-to-one,  and $(Y_{m,n}, \eta_{m,n})$ is primitive.
\end{proof} 
\section{Formulas for the surfaces}
\label{sect:formulas}
In this section, we establish part (3) of Proposition \ref{prop:formulas}. The proofs of the remaining parts are almost identical. In addition, this final formula is the only one not discussed in
\cite{BM}. Therefore, we will assume both $m$ and $n$ are even in this section.

The formulas in this proposition are an application of the Schwarz-Christoffel Mapping Theorem. 
For background see \cite{DT02}, for instance. This formula has been used in the past to find formulas
for translation surfaces which arise from billiard tables \cite{AI88} \cite{Ward}. 


Fix a choice of $m$ and $n$ even. The surface $(Y^e_{m,n},\eta^e_{m,n})$ admits $n$ affine automorphisms whose derivatives are Euclidean reflections.
(See Proposition \ref{prop:orthogonal2}.)
Let $\sC$ be the set of all closures of the 
connected components of the complement of the union of fixed points of these automorphisms. One component is shown in Figure \ref{fig:reflection_domain}. There are a total of $2n$ components, and each is a simply connected polygonal subset of $(Y^e_{m,n}, \eta^e_{m,n})$ with $m/2+2$ vertices and edges. 

Let $C_0, \ldots, C_{m/2-1}$ denote the center points of the polygons $P(0), \ldots, P(\frac{m}{2}-1)$, respectively. Each connected component 
has the points $C_0, \ldots, C_{m/2-1}$ as vertices. The other two vertices
are a midpoint $M$ of an edge of polygon $P(\frac{m}{2}-1)$, and a singularity $S$.
The components come in two different possible orientations. Let $D \in \sC$ be a component so that as we travel around $\partial D$ counter-clockwise we
visit vertices in the following order:
$$S \to C_0 \to C_1 \to \ldots \to C_{m/2-1} \to M \to S.$$
Figure \ref{fig:reflection_domain} shows a possible $D \in \sC$.

Throughout this section we will use $U$ and $L$ to denote the closures
of the upper and lower half-planes in the Riemann sphere $\widehat \C$. We also define
\begin{equation}
\label{eq:uk}
u_k=2 \cos \left(\pi-\frac{(2k+1)\pi}{m}\right) \quad \text{for $k=0,\ldots,\frac{m}{2}-1$.}
\end{equation}

As an application of 
the Schwarz-Christoffel Mapping Theorem, we have:

\begin{figure}[t]
\begin{center}
\includegraphics[width=5in]{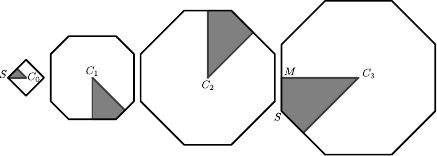}
\caption{This is the surface $(Y^e_{8,4}, \eta^e_{8,4})$. The region bounded by the gray lines is a component of the complement of the collection
of fixed points for affine automorphisms whose derivatives are Euclidean reflections.}
\label{fig:reflection_domain}
\end{center}
\end{figure}

\begin{lemma}
\label{lem:formula}
There is a complex constant $K$ so that the function
$f: U \to \C$ defined by
$$f(u)=K \int_{i}^u (\zeta-2)^{-\frac{1}{2}} \prod_{k=0}^{\frac{m}{2}-1} (\zeta-u_k)^{\frac{1}{n}-1}~d\zeta$$
lifts to a holomorphic bijection $\til f:U \to D \subset Y_{m,n}^e$
so that $\eta_{m,n}^e$ restricted to $D$ is the pullback of $dz$ under the map 
$z=f \circ \til f^{-1}:D \to \C$.
For this lift we have $\til f(\infty)=S$,
$\til f(2)=M$ and
$\til f(u_k)=C_k$ for each $k=0,\ldots,\frac{m}{2}-1$.
\end{lemma}

We prove this lemma in subsection \ref{sect:mapping}.

\subsection{Proof of Proposition \ref{prop:formulas} given the Lemma}
We will now explain how the above lemma leads to the formula in part (3) of Proposition \ref{prop:formulas}.

Consider the mapping $\til f:U \to D$ as in the lemma. We can apply the Schwarz Reflection Principle to extend this map.  
Choose any $D' \in \sC$ which shares an edge $e$ with $D$. By the Schwarz Reflection Principle, we can extend $\til f$ to a map
to $D \cup D'$. This extension is defined on $U \sqcup L/ f^{-1}(e)$, the disjoint union of the closed upper half plane $U$ and the closed lower half plane $L$ 
identified along $f^{-1}(e) \subset \R$. For $u \in L$, we define
$$\til f(u)=R_e \circ \til f(\overline{u}),$$
where $R_e:Y_{m,n}^e \to Y_{m,n}^e$ denotes the reflection in $e$. We also can extend the original map $f:U \to \C$ in a similar way. We obtain a map $f:U \sqcup L/ f^{-1}(e)
\to \C$ whose image in $\C$ is the developed image of $D \cup D'$ obtained by integrating $\eta^e_{m,n}$ over this union. With this definition, we have that 
$\eta^e_{m,n}$ restricted to $D \cup D'$ is the pullback of $dz$ under the map $z=f \circ \til f^{-1}$.

Now consider all of $\sC$. 
The combinatorics of the edge identifications of the components making up $\sC$ can be recorded using a graph. Here we have a vertex for each component in $\sC$ and we draw an edge between two vertices for each edge shared by the corresponding components. Choose a spanning tree for this graph. We inductively apply the Schwarz Reflection Principle along the edges
of the components associated to the edges of the tree. The end result is an extension $\til f$ from a simply connected space $X$, which is a disjoint union of multiple closed upper and lower half planes with identifications by intervals in $\R$, onto the surface $Y^e_{m,n}$. Now
$\eta^e_{m,n}$ is the pullback of $dz$ under the map $z=f \circ \til f^{-1}$ globally.
(There may be multiple choices for a preimage under $\til f$, but each choice satisfies this condition.)

The map $\til f:X \hookrightarrow Y^e_{m,n}$ is one-to-one except on $\til f(\partial X)$. Define
$X'=X/\sim$, where the equivalence relation is defined so that $x_1 \sim x_2$ if $\til f(x_1)=\til f(x_2)$. By definition, $\til f$ descends to a biholomorphic map
$\widehat{f}:X' \to Y^e_{m,n}$. 
Furthermore, $X'$ is branched cover of $\widehat{\C}$. Let $\pi:X' \to \widehat \C$ denote this covering map. By construction, the cover automorphisms of $X'$ are conjugate under $\widehat{f}$ to the action
of $SO(Y^e_{m,n}, \eta^e_{m,n})$ on $Y^e_{m,n}$ by affine automorphisms.

We define two complex valued functions on $X'$. Namely, for $x \in X'$ we define
$$u(x)=\pi(x) \quad \text{and}\quad  w(x)=\frac{1}{f'(x)}.$$
Using the lemma, we see that $w$ and $u$ are related on the upper half plane $\til f^{-1}(D)$ by
$$\frac{1}{w}=K (u-2)^{-\frac{1}{2}} \prod_{k=0}^{\frac{m}{2}-1} (u-u_k)^{\frac{1}{n}-1}.$$
Therefore, globally we have the relationship
\begin{equation}
\label{eq:alg1}
w^n=\frac{1}{K^n} (u-2)^\frac{n}{2} \prod_{k=0}^{\frac{m}{2}-1} (u-u_k)^{n-1}.
\end{equation}
Let $\omega$ denote the pullback of the one form $\eta^e_{m,n}$ under $\widehat{f}^{-1}$.
Since $\eta^e_{m,n}$ is the pullback of $dz$ under $z=f \circ \til f^{-1}$, we have 
\begin{equation}
\label{eq:alg2}
\omega=f'(u)du=\frac{du}{w}.
\end{equation}
We can now observe that points on $X'$ are uniquely determined by their $u$ and $w$ coordinates.
The $u$-coordinate determines the value of the covering projection $\pi:X' \to \C$. For any two distinct points $x_1,x_2 \in X'$ with $u(x_1)=u(x_2)$ we observe that $w(x_2)/w(x_1)$ is an $n$-th root of unity 
recording the element of $SO(Y^e_{m,n}, \eta^e_{m,n})$ whose associated affine automorphism
sends $\widehat{f}(x_1)$ to $\widehat{f}(x_2)$. So the translation surface $(Y^e_{m,n}, \eta^e_{m,n})$ 
is the same as the Riemann surface defined in equation \ref{eq:alg1} equipped with the holomorphic 1-form defined in equation \ref{eq:alg2}. 

The equations given above are not exactly the same as the ones provided in part (3) of Proposition \ref{prop:formulas}. As an alternate coordinate scheme, we define $y$ as a replacement for $w$:
$$y=\frac{(u-2) \prod_{k=0}^{\frac{m}{2}-1} (u-u_k)}{Kw}.$$
After a little algebra, we observe that 
$$y^n=(u-2)^\frac{n}{2} \prod_{k=0}^{\frac{m}{2}-1} (u-u_k)
\quad \text{and} \quad
\omega=\frac{K y~du}{(u-2) \prod_{k=0}^{\frac{m}{2}-1} (u-u_k)}
.$$
This is the same expression as in part (3) of Proposition \ref{prop:formulas}, except for the presence
of the constant $K$ in our expression for the $1$-form. The removal of $K$ has the effect of scaling and rotating the translation surface.

\subsection{The Schwarz-Christoffel Mapping}
\label{sect:mapping}
The goal of this subsection is to prove the lemma above. The lemma is a consequence of the Schwarz-Christoffel Mapping Theorem. 
We will state a variant of this theorem for non-embedded polygons. A {\em generalized polygon} is just a loop formed by a finite list of line segments in the plane. We'll say that a {\em polygonal immersion} is an immersion $\phi$ from a closed disk $\Delta$ into the plane $\C$ whose restriction to $\partial \Delta$ yields a generalized $n$-gon $P$ with vertices $v_0, \ldots, v_{n-1}$. At each vertex $v_j$ we have a well defined notion of an interior angle $\alpha_j>0$. By the Gauss-Bonnet Theorem these $n$ angles sum to $(n-2) \pi.$

\begin{theorem}[Schwarz-Christoffel Mapping Theorem]
\label{thm:sc}
Let $z_0 < z_2< \ldots < z_{n-2}$ be real numbers and let $z_{n-1}=\infty$, considered as a point on the Riemann sphere $\widehat{\C}$.
Choose positive real numbers $\alpha_0, \alpha_1, \ldots, \alpha_{n-1}$ 
whose sum is $(n-2)\pi$.
Let $U$ be the closure of upper half plane in $\widehat{\C}$ and define the function
$$g:U \to \C; \quad f(z)= \int_i^z \prod_{j=0}^{n-2} (\zeta-z_j)^{
(\alpha_j/\pi)-1}~d\zeta.$$
Then, $g$ is a polygonally immersion so that $g(\partial U)$ is a generalized $n$-gon with vertices
$v_j=f(z_j)$ and interior angles $\alpha_j$ for $j=0,\ldots,n-1$.
\end{theorem}

We can use the the Schwarz-Christoffel Mapping Theorem to easily check that the angles of $D$ agree with the angles of $f(U)$.
There are two points that need further attention to prove Lemma \ref{lem:formula}. 

First, we need to show that the polygon $f(\partial U)$ is similar to 
$\dev(\partial D)$, where $\dev:D \to \C$ a the developing map to the plane defined by integrating the holomorphic $1$-form $\eta^e_{m,n}$. 

Second, the lemma claims we can lift the map $f$ to a map to $D \subset Y_{m,n}$. That we can do this is not obvious. It is well known that there are immersed curves in the plane which bound multiple immersed disks. The most famous example is Milnor's Doodle; see \cite[pp. 315]{MilnorCollected3}.
We have
two immersions of closed disks $f:U \to \C$ and $\dev:D \to \C$ which agree on the boundaries of the disks. Saying that
$f$ lifts to a map $\til f:U \to D$ is equivalent to saying that these
two immersed disks are the same (up to precomposition with a homeomorphism). So, if the generalized polygon
$\dev(\partial D)$ bounds multiple disks, we need an additional argument
to obtain the lift.

Fortunately, this second concern is moot: the generalized polygon  $\dev(\partial D)$ only bounds one immersed disk. The polygonal immersion $\dev:D \to \C$ is an example of a ``parking garage\footnote{The author first heard this term used by M. M\"oller.},''
which we define in the next paragraph.

We will say a {\em spiral} is a piecewise differentiable curve
 $\gamma:[0,1] \to \C \smallsetminus \{0\}$
so that 
$\frac{d}{dt} \arg \gamma(t)>0$ for all $t$. To form a closed curve $C$ we augment $\gamma$ to begin with the line segment from $0$ to $\gamma(0)$ and end with the line segment from $\gamma(1)$ to $0$. We will observe that such a curve bounds an immersed disk. We define a closed topological disk
$$\Delta=\{(x,y) \in \R^2~:~ \text{$0 \leq x \leq 1$ and $0 \leq y \leq x$}\}.$$ 
Then we define the immersion $\phi:\Delta \to \C$ by $\phi(x,y)=x \gamma(y/x).$ Observe that $\phi(\partial \Delta)=C$. We call any immersion of a closed disk
into $\C$ a {\em parking garage}, if it differs from a map of the form $\phi:\Delta \to \C$
by precomposition with a homeomorphism. 

\begin{proposition}
\label{prop:lift}
Let $\gamma$ be a spiral. Then, the associated immersion $\phi:\Delta \to \C$ 
is the only immersed disk with boundary $\phi(\partial \Delta)$. That is,
if $\psi:\Delta \to \C$ is another immersion
so that $\psi$ agrees with $\phi$ pointwise on $\partial \Delta$, then there is a homeomorphism
$h:\Delta \to \Delta$ preserving $\partial \Delta$ and satisfying
$\psi = \phi \circ h$.
\end{proposition}

\begin{proof}
Suppose $\psi$ is as in the statement of the proposition. We will construct $h$.

Consider the universal covering map $\pi:X \to \C \smallsetminus \{0\}$. Let $p=(1,0)$ in the boundary of $\Delta$.
Fix $w=\phi(p)=\psi(p) \in \C \smallsetminus \{0\}$
and a lift $\widetilde w \in X$. 
Set $\Delta_0=\Delta \smallsetminus \{(0,0)\}$. Then we have unique lifts $\widetilde \phi, \widetilde \psi:\Delta_0 \to X$ so that $\widetilde \phi(p)=\widetilde \psi(p)=\widetilde w$.

Let $\overline X$ be the one point compactification of $X$. Observe that $\overline X$ is a sphere.
There are a unique continuous extensions $\overline \phi, \overline \psi:\Delta \to \overline X$ of $\widetilde \phi$ and $\widetilde \psi$, respectively. These extensions just send $(0,0)$ to the point added to build $\overline X$.
The images of $\partial \Delta$ under $\overline \phi$ and $\overline \psi$ agree. This is a simple closed curve $\widetilde \gamma$ in $\overline X$, which can be explicitly described in terms of $\gamma$. By the Jordan Curve Theorem, we know $\overline X \smallsetminus \widetilde \gamma$ consists of two disks. The maps $\overline \phi$ and $\overline \psi$ must send $\Delta$ to the same disk, because only one of the disks has bounded image under the map $\pi$. 
Therefore, there is an $h:\Delta \to \Delta$ so that $\widetilde \psi = \widetilde \phi \circ h$.
By construction, the same $h$ satisfies the proposition.
\end{proof}

We will now summarize the ideas we use to Lemma \ref{lem:formula}.
Recall that by the Schwarz-Christoffel Mapping Theorem, the image
of the map $f$ defined in the Lemma is an immersed disk whose boundary is a generalized polygon whose angles agree with the boundary of the immersed disk 
obtained by developing $D \subset Y^e_{m,n}$ into the plane using the $1$-form $\eta^e_{m,n}$. The identification of vertices is as stated in the Lemma. 
Observe that the developing map applied to $D$ is a parking garage (up to postcomposition with a homeomorphism). Because of the above proposition,
to show that $f$ lifts to a map $\til f$, it suffices to show that the generalized polygon $f(\partial U)$ is similar to the generalized polygon obtained by developing $\partial D$ into the plane. We know the angles agree. To prove that there is a similarity, it will typically suffice to check that all but two 
of the edge lengths of these two polygons agree up to a scalar constant. See the formal proof below.

\begin{proof}[Proof of lemma \ref{lem:formula}.]
We may develop $D$ into the plane using the holomorphic $1$-form $\eta_{m,n}^e$. The image of $\partial D$ is a generalized polygon,
which we call $P$. We abuse notation by naming the vertices of $P$ in the same way as the vertices of $\partial D$. Observe
that $P$ is a generalized $\frac{m}{2}+2$-gon. We think of $m$ as fixed, and allow $n$ to vary. 
The angle of $P$ at each $C_k$ is $\frac{\pi}{n}$ for $k=0,\ldots, \frac{m}{2}-1$. And the angle at $M$ is $\frac{\pi}{2}$. 
Set $\kappa=(\cos \frac{\pi}{m}+\cos \frac{\pi}{n})/ \sin \frac{\pi}{n}$.
All the edge lengths of $P$ are trigonometric expressions in $m$ and $n$:
$$|C_{k-1} C_k|=\kappa \sin \frac{k \pi}{m}, \quad \text{for $k=1, \ldots, \frac{m}{2}-1$.}$$
$$|C_{m/2-1} M|=\frac{1}{2}\kappa. \qquad |MS|=\frac{1}{2}. \qquad |S C_0|=\frac{\sin \frac{\pi}{m}}{2 \sin \frac{\pi}{n}}.$$
These quantities can be computed by working with the definition of $(Y_{m,n}^e,\eta_{m,n}^e)$ as a union of
semiregular polygons. 
We consider $n \geq 2$ to be a real number, for reasons that will become apparent at the end of the proof. We can define the generalized polygon $P$ for each real $n \geq 2$, because we have given the angles
and the edge lengths as functions of $n$. (It is an exercise to check that these
angles and edge lengths always lead to a closed generalized polygon when $n \geq 2$.)

Similarly, we define $Q$ to be the polygon $f(\partial U)$. By the Schwarz-Christoffel Theorem, this is always a 
generalized $\frac{m}{2}+2$-gon with vertices $f(u_k)$ for $k=0,\ldots, \frac{m}{2}-1$, $f(2)$ and $f(\infty)$. 
Again this polygon is well defined for all real $n \geq 2$.

We will show that the polygons $P$ and $Q$ differ by a similarity for each $n \geq 2$. The angles match by the Schwarz-Christoffel Theorem. We will now check that certain corresponding pairs of edges are proportional
in length. 

Let $\ell_k$ denote the length of the line segment
$f([u_{k}, u_{k-1}])$ for $k=1,\ldots,\frac{m}{2}-1$. This length is expected to be proportional to the distance from $C_{k-1}$ to $C_{k}$, namely $\kappa \sin \frac{k \pi}{m}$. By definition of $f$,
$$\ell_k=|f(u_{k})-f(u_{k-1})|=\Big| K \int_{u_{k}}^{u_{k-1}}(\zeta-2)^{-\frac{1}{2}} \prod_{k=0}^{\frac{m}{2}-1} (\zeta-u_k)^{\frac{1}{n}-1}~dz \Big|.$$
Now, we make the change of coordinates $\zeta=2 \cos(2 w)$. After some trigonometric manipulations, we see
\begin{equation*}
\ell_k=2^{\frac{n+1}{n}} |K| \Big|\int_{\frac{m-2k+1}{2m}\pi}^{\frac{m-2k-1}{2m}\pi} \frac{\cos w}{\cos(mw)^{1-\frac{1}{n}}}~dw\Big|.
\end{equation*}
Now make a change of coordinates $x=w+\frac{k\pi}{m}-\frac{\pi}{2}$. This yields
$$
\begin{array}{rcl}
\ell_k & = & \displaystyle 2^{\frac{n+1}{n}} |K| \Big|\int_{\frac{-\pi}{2m}}^{\frac{\pi}{2m}} \frac{\cos (x+\frac{\pi}{2}-\frac{k\pi}{m}) }{[\pm\cos(mx)]^{1-\frac{1}{n}}}~dx\Big| \\
 & = &
\displaystyle 2^{\frac{n+1}{n}} |K| \Big|\int_{\frac{-\pi}{2m}}^{\frac{\pi}{2m}}
 \frac{\cos (x)\cos(\frac{\pi}{2}-\frac{k\pi}{m})-\sin (x)\sin(\frac{\pi}{2}-\frac{k\pi}{m})}{\cos(mx)^{1-\frac{1}{n}}}~dx\Big|
\end{array} 
$$
The sine term drops out by symmetry. Also, $\cos(\frac{\pi}{2}-\frac{k\pi}{m})=\sin(\frac{k\pi}{m}).$ Thus we have 
\begin{equation*}
\ell_k = 
2^{\frac{n+1}{n}}  \sin(\frac{k\pi}{m}) |K| \int_{\frac{-\pi}{2m}}^{\frac{\pi}{2m}} \frac{\cos (x)}{\cos(mx)^{1-\frac{1}{n}}}~dx.
\end{equation*}
The integral is now independent of $k$, so we see that $\ell_k$ is proportional to $\sin(\frac{k\pi}{m})$. Thus these edges $f([u_{k-1}, u_{k}])$
are proportional in length to the distance from $C_{k-1}$ to $C_{k}$,
as desired. The proportionality constant is
\begin{equation}
\label{eq:proportionality}
\kappa^{-1} 2^{\frac{n+1}{n}} |K| \int_{\frac{-\pi}{2m}}^{\frac{\pi}{2m}} \frac{\cos (x)}{\cos(mx)^{1-\frac{1}{n}}}~dx.
\end{equation}

We will now discuss one more edge. (We will be leaving out the edges with the singularity $S$ as an endpoint). Let $\ell_0=|f(u_{m/2-1})-f(2)|$.
We need to show that this length differs from the length of the edge
joining $C_{m/2-1}$ to $M$ by the same proportionality constant. The length of the segment $\overline{C_{m/2-1} M}$ is $\frac{1}{2}\kappa$, as remarked above.
Now we will compute $\ell_0$. 
$$\ell_0=|f(u_{m/2-1})-f(2)|=
\Big| K \int_{2}^{u_{m/2-1}}(\zeta-2)^{-\frac{1}{2}} \prod_{k=0}^{\frac{m}{2}-1} (\zeta-u_k)^{\frac{1}{n}-1}~dz \Big|.$$
After making the coordinate change $\zeta=2 \cos(2x)$, we have
$$\ell_0=2^{\frac{n+1}{n}} |K| \big|\int_0^{\frac{\pi}{2m}} \frac{\cos x ~dw}{\cos(mx)^{1-\frac{1}{n}}}\big|=2^{\frac{n+1}{n}} \frac{|K|}{2} \int_{\frac{-\pi}{2m}}^{\frac{\pi}{2m}} \frac{\cos (x)}{\cos(mx)^{1-\frac{1}{n}}}~dx.$$
We observe that this quantity is proportional to the distance from $C_{m/2-1}$ to $M$
for each $n \geq 2$
by the proportionality constant given in Equation \ref{eq:proportionality}.

We conclude by arguing that checking the ratio of these pairs of edges suffices to show the two generalized polygons are similar. We have shown that there is a single similarity carrying the edges $\overline{C_{k-1} C_k}$ to $f([u_{k-1},u_k])$ and carrying $\overline{C_{m/2-1} M}$ to $f([u_{m/2-1}, 2])$, because these edges are proportional in length and the angles between the edges are identical. The angles at the beginning and end of these chains of edges agree as well; the angle at $M$ agrees with that at $f(2)$
and the angle at $C_0$ agrees with that at $f(u_0)$. So we can {\em typically} characterize the location of $S$
by extending lines out of the points $M$ and $C_0$ in the right directions. The point $S$ lies at the intersection of these lines.
In this case, because everything else agrees, the similarity must carry $S$ to $f(\infty)$. 
But it can happen that these two lines coincide. This happens when the interior angle at $S$ is a multiple of $\pi$. We compute that the 
angle at $S$ is given by $\frac{mn-m-n}{2n}\pi$. So, we get our similarity unless $m$ is an odd multiple of $n$. So, it
happens except at a finite list of values of $n \geq 2$. By continuity of the lengths of edges in both generalized polygons, 
we find that this similarity must exist even when $m$ is an odd multiple of $n$. 

By fiddling with the complex constant $K$ we can make it so that the immersed disks
$f(U)$ and the developed image of $D$ under $\eta^e_{m,n}$ differ by a translation.
We know these regions have the same generalized polygon for their boundaries.
Proposition \ref{prop:lift} implies that we can find the lift $\widetilde f$. 
\end{proof}

\section{Lattices which are not Veech groups}
\label{sect:not_veech}

The trace field of a subgroup $\Gamma \subset \PSL(2, \R)$ is the field
$\Q(\tr~\Gamma)=\Q(\tr~\gamma~:~ \gamma \in \Gamma)$. 
Note that the trace of an element of $\PSL(2, \R)$ is only defined up to sign, but these choices have no effect on the definition of this field.
Let $\Gamma^{(2)}=\langle \gamma^2 ~|~\gamma \in \Gamma\rangle$. This is a finite index subgroup of $\Gamma$.
The {\em invariant trace field of $\Gamma$} is the field $\Q(\tr~\Gamma^{(2)})$. To abbreviate notation we use $k \Gamma$ to denote
$\Q(\tr~\Gamma^{(2)})$.
The name is justified by the fact that if $\Gamma$ is finitely generated non-elementary group, 
then $k \Gamma'=k \Gamma$ for any finite index subgroup $\Gamma'  \subset \Gamma$.
In particular, $k \Gamma \subset \Q(\tr~\Gamma') \subset \Q(\tr~\Gamma)$. See \cite[\S 3]{MR03} for further background on this subject.

\begin{lemma}
\label{lem:trace_field}
Suppose that $\Gamma=\PSL(X, \omega)$ is finitely generated and non-elementary. Then $k \Gamma = \Q(\tr~\Gamma)$.
\end{lemma}

This lemma follows from results of Kenyon and Smillie \cite{KS}.
Hubert and Schmidt realized that the following is implied by Theorem 28 of \cite{KS}.

\begin{theorem}[Kenyon-Smillie]
\label{thm:KS}
Let $(X, \omega)$ be a translation surface.
If $A \in \PSL(X, \omega)$ is hyperbolic, then $\PSL(X, \omega)$ is conjugate into $\PSL\big(2, \Q(\tr ~ A)\big)$.
\end{theorem}
This theorem was used in \cite[remark 7]{HS01} to show that $\Delta^+(2, m,\infty)$ can not 
arise as $\PSL(X, \omega)$ when $m$ is even.
The weaker statement of lemma \ref{lem:trace_field} is all that is necessary for our proof that certain triangle groups can not arise as a $\PSL(X, \omega)$,
and the full strength of this theorem of Kenyon and Smillie does not exclude any additional triangle groups. We have chosen the proof using lemma \ref{lem:trace_field}
because it yields a more conceptually natural proof.

\begin{proof}[Proof of Lemma \ref{lem:trace_field}]
Let $B \in \Gamma=\PSL(X, \omega)$ be hyperbolic. $B^2$ is also hyperbolic. 
Thus,
$$\Q(\tr~B^2) \subset k \Gamma \subset \Q(\tr~\Gamma) \subset \Q(\tr ~ B^2),$$
with the last containment following from applying theorem \ref{thm:KS} with $A=B^2$. 
\end{proof}

\begin{remark}[A second proof of lemma \ref{lem:trace_field}]
Another method of proving lemma \ref{lem:trace_field} would be to generalize work of Gutkin and Judge \cite{GJ00}.
Their work implies that if $\Gamma=\PSL(X, \omega)$ contains a finite index subgroup $\Gamma' \subset \Gamma$
which is conjugate into $\PSL(2, \Q)$ then $\Gamma$ must be conjugate into $\PSL(2, \Q)$ as well.
Their proof works with $\Q$ replaced by any subfield of $\R$. 
\end{remark}

With regard to triangle groups, we have the following.

\begin{lemma}
\label{lem:triangle_groups_fields}
Suppose $2 \leq m \leq n<\infty$, $n>2$, and let $\Gamma_{m,n}=\Delta^+(m,n,\infty) \subset \PSL(2, \R)$. Then $k\Gamma_{m,n}=\Q(\tr~\Gamma_{m,n})$ unless
one of the following statements holds.
\begin{enumerate}
\item $\gcd(m,n)=2$. 
\item Both $m$ and $n$ are even, and both $m/\gcd(m,n)$ and $n/\gcd(m,n)$ are odd.
\end{enumerate}
\end{lemma}

Given this lemma, the proof of theorem \ref{thm:non_veech_triangle_groups} of the introduction follows
by concatenating lemmas \ref{lem:trace_field} and \ref{lem:triangle_groups_fields}. The remainder of this section will be devoted to proving
lemma \ref{lem:triangle_groups_fields}. We will heavily use the book of Maclachlan and Reid \cite{MR03}.

The group $\Gamma_{m,n}=\Delta^+(m,n,\infty) \subset \PSL(2, \R)$ is generated by the projections of the following matrices to $\PSL(2, \R)$.
\begin{equation}
\label{eq:matrices1}
X=\left[\begin{array}{rr}
0 & -1 \\
1 & 2 \cos \frac{\pi}{m} \\
\end{array}\right]
\quad 
Y=\left[\begin{array}{rr}
-2 \cos \frac{\pi}{n} & 1 \\
-1 & 0
\end{array}\right]
\end{equation}
These matrices satisfy the identities $X^m=Y^n=-I$ (which projects to the identity in $\PSL(2, \R)$, while $XY$ is parabolic. 

The following follows from lemma 3.5.3 of \cite{MR03}.
\begin{proposition}[Trace field]
$\displaystyle \Q(\tr~\Gamma_{m,n})=\Q(\cos \frac{\pi}{m}, \cos \frac{\pi}{n})$. 
\end{proposition}
Note that if $m=2$, then $\Q(\tr~\Gamma_{m,n})=\Q(\cos \frac{\pi}{n})$. 
The following follows from lemmas 3.5.7 and 3.5.8 of \cite{MR03}.

\begin{proposition}[Invariant trace field]
 $\displaystyle k\Gamma_{m,n} = \Q(\cos \frac{2\pi}{m}, \cos \frac{2\pi}{n}, \cos \frac{\pi}{m} \cos \frac{\pi}{n})$. 
\end{proposition}
Note that in the special case that $m=2$, we have $k\Gamma_{m,n} = \Q(\cos \frac{2\pi}{n})$. 

\begin{proposition}
$[\Q(\tr~\Gamma_{m,n}):k\Gamma_{m,n}] \leq 2$. Let $p,q \in k\Gamma_{m,n}[x]$ denote the polynomials
$$p(x)=2x^2+(\cos \frac{2\pi}{m}-1)
\quad \textrm{and} \quad
q(x)=2x^2+(\cos \frac{2\pi}{n}-1).$$
\begin{enumerate}
\item In the case $m=2$, $\Q(\tr~\Gamma_{m,n})$ is the splitting field of $q$.
\item Otherwise, $\Q(\tr~\Gamma_{m,n})$ is both the splitting field for $p$ and the splitting field for $q$.
\end{enumerate}
\end{proposition}
\begin{proof}
By the double angle formula, $\cos \frac{\pi}{m}$ and $\cos \frac{\pi}{n}$ are roots of $p$ and $q$, respectively. 
Combined with the knowledge that $\Q(\tr~\Gamma_{2,n}) = \Q(\cos \frac{\pi}{n})$ and $k\Gamma_{2,n} = \Q(\cos \frac{2\pi}{n})$,
this implies statement (1). Note that $\cos \pi/m \cos \pi/n \in k \Gamma_{m,n}$. Thus when $m \neq 2$, $\cos \frac{\pi}{m} \in k \Gamma_{m,n}$
implies $\cos \frac{\pi}{n} \in k \Gamma_{m,n}$, and vice versa. This implies statement (2).
\end{proof}

We break the remainder of the proof of lemma \ref{lem:triangle_groups_fields} into special cases.

\begin{corollary}[The case $m=2$]
\label{cor:m2}
$k\Gamma_{2,n}=\Q(\tr~\Gamma_{2,n})$ if and only if $n$ is odd.
\end{corollary}
\begin{proof}
Let $\zeta_n=e^{i \pi/n}$. We have that $[\Q(\zeta_{n}^2):k\Gamma_{2,n}]=[\Q(\zeta_n):\Q(\tr~\Gamma_{2,n})]=2$,
as our fields of interest are the real subfields of cyclotomic fields. We note that 
$$[\Q(\zeta_{n}^2):\Q]=\varphi(n) 
\quad \textrm{and} \quad
[\Q(\zeta_{n}):\Q]=\varphi(2n),$$
where $\varphi$ denotes the Euler $\varphi$ function. From the standard formula for $\varphi$, we have
that
$$[\Q(\tr~\Gamma_{2,n}):k\Gamma_{2,n}]=\frac{\varphi(2n)}{\varphi(n)}=\begin{cases}
1 & \textrm{if $n$ is odd} \\
2 & \textrm{if $n$ is even.}
\end{cases}$$
\vspace{-2em}\\
\end{proof}

\begin{corollary}[The odd case]
If either $m$ or $n$ is odd, then $k\Gamma_{m,n}=\Q(\tr~\Gamma_{m,n})$.
\end{corollary}
\begin{proof}
If $m$ is odd then the polynomial $p$ splits over $\Q(\cos 2\pi/m)$ and hence over $k\Gamma_{m,n}$.
(A proof equivalent to that of corollary \ref{cor:m2} applies.)
Similarly, when $n$ is odd, $q$ splits over $\Q(\cos 2\pi/n)$.
\end{proof}

The following finishes the proof of lemma \ref{lem:triangle_groups_fields}.

\begin{proposition}[The even case]
Assume $m>2$ and both $m$ and $n$ are even. Then $k\Gamma_{m,n}=\Q(\tr~\Gamma_{m,n})$
unless $\gcd(m,n)=2$ or both $m/\gcd(m,n)$ and $n/\gcd(m,n)$ are odd.
\end{proposition}
\begin{proof}
Let $x=mn/\gamma$ and $\zeta=e^\frac{2\pi i}{2x}$, so $\zeta$ is a $2x$-th root of unity. 
We consider the cyclotomic field $\Q(\zeta)$. Note that
$$2 \cos \frac{\pi}{m}=\zeta^{x/m}+\zeta^{-x/m}
\quad \textrm{and} \quad
2 \cos \frac{\pi}{n}=\zeta^{x/n}+\zeta^{-x/n}.$$
The Galois automorphisms of $\Q(\zeta)$ over $\Q$ are all induced by $\zeta \mapsto \zeta^k$ where 
$k$ is an integer with $\gcd(k,2x)=1$ and $1 \leq k < 2x$. We denote this Galois automorphism
by $\sigma_k$. 

The field $\Q(\zeta)$ contains both $k\Gamma_{m,n}$ and $\Q(\tr~\Gamma_{m,n})$. 
Assume that $k\Gamma_{m,n} \neq \Q(\tr~\Gamma_{m,n})$. 
Then, we know
$[\Q(\tr~\Gamma_{m,n}):k\Gamma_{m,n}]=2$. Hence, there is a unique non-trivial Galois automorphism
$\sigma \in \mathit{Aut}_{k\Gamma_{m,n}}\Q(\tr~\Gamma_{m,n})$ \cite[corollary V.4.3]{Hu74}. 
By the fundamental theorem of Galois theory, this $\sigma$ is the restriction of a Galois
automorphism of $\Q(\zeta)$ over $\Q$. In particular, it must be that
$\sigma=\sigma_k|_{\Q(\tr~\Gamma_{m,n})}$ for some $k$ with $\gcd(k,2x)=1$ and $1 \leq k < 2x$.
Such an automorphism must be an involution satisfying
\begin{equation}
\label{eq:galois_action}
\sigma_k(\cos \frac{\pi}{m})=-\cos  \frac{\pi}{m}
\quad \textrm{and} \quad
\sigma_k(\cos \frac{\pi}{n})=-\cos  \frac{\pi}{n},
\end{equation}
as it must act transitively on the roots of both $p$ and $q$. Note that any $\sigma_k$ satisfying
equation \ref{eq:galois_action} fixes all elements of $k\Gamma_{m,n}$, but acts non-trivially on elements
of $\Q(\tr~\Gamma_{m,n})$. In particular, $k\Gamma_{m,n} \neq \Q(\tr~\Gamma_{m,n})$
if and only if there is a Galois automorphism satisfying equation \ref{eq:galois_action}.

Consider the set of all $k$ for which $\sigma_k(2 \cos \frac{\pi}{m})=-2 \cos \frac{\pi}{m}$. This is equivalent to saying
that $\sigma_k(\zeta^{x/m})=\zeta^{x\pm x/m}$. In particular, this implies that
$$k \equiv \frac{x \pm x/m+2ax}{x/m} \pmod{2x}$$
for some integer $a$. Similarly, $\sigma_k(2 \cos \frac{\pi}{n})=-2 \cos \frac{\pi}{n}$ implies
$$k \equiv \frac{x \pm x/n+2bx}{x/n} \pmod{2x}$$
for some $b$. By simplifying, we see this is equivalent to the conditions that
$$k \equiv m \pm 1+2am \pmod{2x}
\quad \textrm{and} \quad
k \equiv n \pm 1+2bn \pmod{2x}.$$
The existence of such a $k$ is equivalent to the statement that there are choices of $a, b \in \Z$ and $\epsilon \in \{-2, 0, 2\}$ for which
$$m-n+\epsilon \equiv 2bn-2am \pmod{2x}.$$
Since $\gamma=\gcd(m,n)$, the right hand side can be any even multiple of $\gamma$. Therefore, this is equivalent to the statement that
\begin{equation}
\label{eq:number_theory}
m-n+\epsilon \equiv 0 \pmod{2 \gamma}.
\end{equation}

We will now check to see when equation \ref{eq:number_theory} holds for various choices of $\epsilon$. 
First assume $\epsilon=0$, then everything is divisible
by $\gamma$, so this equation is equivalent to 
$$
m/\gamma-n/\gamma \equiv 0 \pmod{2}.
$$
This equation is true if and only if both $m/\gamma$ and $n/\gamma$ are odd.

Now assume $\epsilon=\pm 2$. Notice that $m-n$ is always a multiple of $\gamma$. In particular, either
$$m-n \equiv 0 \pmod{2\gamma} 
\quad \textrm{or} \quad
m-n \equiv \gamma \pmod{2\gamma}.$$
As $\epsilon=\pm 2$, the only possible values of $m-n+\epsilon$ are $\pm 2$ or $\gamma \pm 2$ modulo
$2 \gamma$. In particular for equation \ref{eq:number_theory} to be true, we must have 
$\gamma=2$. In this case, there is always a choice of $\epsilon \in \{-2, 0, 2\}$ which makes the equation true.
We choose $\epsilon=0$ if both $m/\gamma$ and $n/\gamma$ are odd, and $\epsilon=\pm 2$ otherwise. (The choice of sign is irrelevant in this case.)
\end{proof}

\vspace{1em}
\noindent
{\bf Acknowledgments. }
Many helpful conversations occurred at MSRI's workshop on ``Topics in Teichm\"uller Theory and Kleinian Groups'' held in November, 2007.
The author would like to thank Matt Bainbridge, Pascal Hubert, Martin M\"oller, Yaroslav Vorobets, and Barak Weiss for these helpful conversations.
The author would especially like to thank John Smillie for realizing that a relatively simple presentation of these surfaces should be possible.
The author thanks Curt McMullen for many helpful comments and for clarifying the history of many ideas appearing here, and Ronen Mukamel for asking helpful questions. Finally, the author would like to thank Anja Randecker and an anonymous referee for pointing out some errors in previous versions.

\bibliographystyle{amsalpha}
\bibliography{/home/pat/active/my_papers/bibliography}

\end{document}